\documentclass[12pt, reqno]{amsart}
\usepackage{amssymb,latexsym,amsmath,amscd,amsfonts}
\usepackage{latexsym}
\usepackage[mathscr]{eucal}
\usepackage{bm}

\voffset = -28pt \hoffset = -54pt \textwidth = 6.3in \textheight = 9.5in \numberwithin{equation}{section}

\newcommand{\beg}{\begin{equation}}
\newcommand{\eeg}{\end{equation}}
\newcommand{\ben}{\begin{eqnarray*}}
	\newcommand{\een}{\end{eqnarray*}}

\newtheorem{thm}{Theorem}[section]
\newtheorem{cor}[thm]{Corollary}
\newtheorem{lem}[thm]{Lemma}

\numberwithin{equation}{section} 
\theoremstyle{definition}
\newtheorem{defn}[thm]{Definition}

\newcommand{\T}{\mathbb{T}}
\newcommand{\ft}{\mathcal F_O}

\newcommand{\M}{\mathcal{M}}

\begin{document}
	\title[Admissible fundamental operators]
{Admissible fundamental operators associated with two domains related to $\mu$-synthesis}

\author[Bappa Bisai]{Bappa Bisai}
\address{Mathematics group, Harish-Chnadra Research Institute, HBNI, Chhatnag Road, Jhunsi, Allahabad, 211019, India.} \email{bappabisai@hri.res.in, bappa.bisai1234@gmail.com}




\begin{abstract}
	A commuting tuple of $n$-operators $(S_1, \dots, S_{n-1},P)$ on a Hilbert space $\mathcal{H}$ is called a $\Gamma_n$-contraction if the closed symmetrized polydisc 
	\[
	\Gamma_n=\left\{\left(\sum\limits_{1\leq i\leq n}z_i, \sum\limits_{1\leq i<j\leq n}z_iz_j, \dots, \prod\limits_{i=1}^nz_i\right): |z_i|\leq 1,\, i=1,\dots, n \right\}
	\]
	is a spectral set. Also a commuting triple of operators $(A,B,P)$ for which the closed tetrablock $\overline{\mathbb E}$ is a spectral set is called an $\mathbb E$-contraction, where 
	\[
	\mathbb E=\left\{ (\beta_1+\bar{\beta}_2x_3,\beta_2+\bar{\beta}_1x_3,x_3)\in \mathbb C^3: |x_3|< 1 \text{ and }|\beta_1|+|\beta_2|< 1 \right\}.
	\]
	To every $\Gamma_{n}$-contraction, there is a unique operator tuple $(A_1, \dots, A_{n-1})$, defined on $\overline{\text{Ran}}D_P$, such that 
	\[
	S_i-S_{n-i}^*P=D_{P}A_iD_P, \quad D_{P}=(I-P^*P)^{1/2}, \quad i=1, \dots, n-1.
	\]
	This is called the \textit{fundamental operator tuple} (or in short the $\mathcal{F}_O$-tuple) associated with the $\Gamma_{n}$-contraction. Similarly, for every $\mathbb{E}$-contraction there is a unique $\mathcal{F}_O$-pair, defined on $\overline{\text{Ran}}D_P$, such that 
	\[
	A-B^*P=D_PF_1D_P, \quad B-A^*P=D_PF_2D_P.
	\]
	In this article, we discuss necessary condition of conditional dilation for both completely non-unitary (c.n.u) $\Gamma_{n}$-contractions and c.n.u $\mathbb E$-contractions.
	Consider two tuples, $(A_1, \dots, A_{n-1})$ and $(B_1, \dots, B_{n-1})$, of operators defined on two Hilbert spaces. One of the principal goals is to identify a necessary and a sufficient condition guaranteeing the existence of a c.n.u $\Gamma_n$-contraction $(S_1, \dots, S_{n-1},P)$ such that $(A_1, \dots, A_{n-1})$ becomes the $\mathcal{F}_O$-tuple of $(S_1, \dots, S_{n-1},P)$ and $(B_1, \dots,B_{n-1})$ becomes the $\mathcal{F}_O$-tuple of $(S_1^*,\dots, S_{n-1}^*,P^*)$. Also for given two pairs of operators $(F_1,F_2)$ and $(G_1,G_2)$ defined on two Hilbert spaces, we examine when there is a c.n.u $\mathbb E$-contraction $(A,B,P)$ such that $(F_1,F_2)$ becomes the $\mathcal{F}_O$-pair of $(A,B,P)$ and $(G_1,G_2)$ becomes the $\mathcal{F}_O$-pair of $(A^*,B^*,P^*)$.
\end{abstract}

\maketitle


\section{Introduction and preliminaries}

Throughout this article all operators are bounded linear transformations defined on separable complex Hilbert spaces.\\

In \cite{Agler}, Agler and Young defined the symmetrized bidisc $\mathbb G_2$ to study the $2 \times 2$ spectral Nevanlinna-Pick interpolation problem. Later Costara in \cite{costara} analyzed the symmetrized polydisc $\mathbb G_n$ to study the generalized $n \times n$ spectral Nevanlinna-Pick interpolation problem, where 
\[
\mathbb G_n =\left\{ \left(\sum_{1\leq i\leq n} z_i,\sum_{1\leq
	i<j\leq n}z_iz_j,\dots, \prod_{i=1}^n z_i \right): \,|z_i|< 1,
i=1,\dots,n \right \}.
\]
In 2007, Abouhajar, White and Young introduced the tetrablock $\Bbb E$ (see \cite{A:W:Y}), where
\[
\mathbb E=\left\{ (\beta_1+\overline{\beta}_2x_3,\beta_2+\overline{\beta}_1x_3, x_3): |x_3|< 1 \text{ and }|\beta_1|+|\beta_2|< 1 \right\}.
\]
The motivation for studying this domain comes from a class of interpolation problems in $H^{\infty}$ control theory. \\

The domains $\mathbb G_n$ and $\mathbb E$ are originated in the $\mu$-synthesis problem. Given a linear subspace $E\subseteq \mathcal{M}_n(\mathbb C)$, the space of all $n \times n$ complex matrices, the structured singular value of $A \in \mathcal{M}_n(\mathbb{C})$ is the number
\[
\mu_E(A)=\left(\text{inf}\{\|X\|:X\in E \text{ and }(I-AX) \text{ is singular}\}\right)^{-1}.
\]
If $E= \mathcal{M}_n(\Bbb C)$, then $\mu_E(A)=\|A\|$, the operator norm of $A$. In this case the $\mu$-synthesis problem is the classical Nevanlinna-Pick interpolation problem. If $E$ is the space of all scalar multiples of the identity matrix $I_{n \times n}$, then $\mu_E(A)=r(A)$, the spectral radius of $A$. For any linear subspace $E$ of $\mathcal{M}_n(\Bbb C)$ that contains the identity matrix, $r(A)\leq \mu_{E}(A)\leq \|A\|$. For the control-theory motivations behind $\mu_E$ an interested reader can see \cite{doyel}. For a given subspace $E\subseteq \mathcal{M}_n(\mathbb C)$, the aim of $\mu$-synthesis problem is to construct an analytic $n \times n$-matrix-valued function $F$ on the open unit disk $\mathbb D$ (with center at the origin) subject to a finite number of interpolation conditions such that $\mu_E(F(\lambda))<1$ for all $\lambda\in \Bbb D$. If $E=\{aI_{n\times n}: a \in \Bbb C\}$, then $\mu_E(A)=r(A)<1$ if and only if $\pi_n(\lambda_1, \dots, \lambda_n)\in \Bbb G_n$ (see \cite{costara}), where $\lambda_1,\dots, \lambda_n$ are eigenvalues of $A$ and $\pi_n$ is the symmetrization map on $\Bbb C^n$ defined by 
\[
\pi_n(z_1, \dots, z_n)=\left(\sum_{1\leq i\leq n} z_i,\sum_{1\leq
	i<j\leq n}z_iz_j,\dots, \prod_{i=1}^n z_i\right).
\]
If $E$ is the linear subspace of $2 \times 2$ diagonal matrices, then for any $A=(a_{ij})\in \mathcal{M}_2(\Bbb C)$, $\mu_E(A)<1$ if and only if $(a_{11},a_{22}, det A)\in \Bbb E$ (see \cite{A:W:Y}). Though the origin of these two domains is control engineering, the domains have been well studied by numerous mathematicians over the past two decades from the perspectives of complex geometry, function theory, and operator theory. An interested reader is referred to some exciting works of recent past \cite{A:W:Y,Agler,costara,E:Z,T:B,PalAdv,B:P,B:P5,B:P6,S:B,S:Pdecom,PalNagy} and references there in. In this article, we shall particularly focus on operator theoretic sides of $\mathbb G_n$ and $\Bbb E$.
\begin{defn}
	A compact subset $X\subset \Bbb C^n$ is said to be \textit{spectral set} for a commuting $n$-tuple of operators $\underline{T}=(T_1, \dots, T_n)$ if the Taylor joint spectrum $\sigma(\underline{T})$ of $\underline{T}$ is a subset of $X$ and von Neumann inequality holds for every rational function, that is, 
	\[
	\|f(\underline{T})\|\leq \|f\|_{\infty, X},
	\]
	for all rational functions $f$ with poles off $X$.
\end{defn}
For a detailed discussion on Taylor joint spectrum, an interested reader is referred to Taylor's works \cite{Taylor, Taylor1} or Curto's survey article \cite{curto}.  
\begin{defn}
	A commuting $n$-tuple of operators $(S_1, \dots, S_{n-1},P)$ for which the closed symmetrized polydisc $\Gamma_{n}(=\overline{\Bbb G}_n)$ is a spectral set is called a \textit{$\Gamma_{n}$-contraction}. In a similar fashion, a commuting triple of operators $(A,B, P)$ for which $\overline{\Bbb E}$ is a spectral set is called an \textit{$\Bbb E$-contraction}.
\end{defn}
 One of the most remarkable inventions in operator theory on these two domains is the existence and uniqueness of \textit{fundamental operator tuples} or \textit{fundamental operator pairs}. The concept of the fundamental operator of a $\Gamma$-contraction was first introduced in \cite{PalAdv}. It is proved in \cite{PalAdv} that to for every $\Gamma$-contraction $(S,P)$ there is a unique operator $A$ on $\mathcal{D}_P$ such that 
 \[
 S-S^*P=D_PAD_P,
 \]
 where $D_P=(I-P^*P)^{1/2}$ and $\mathcal{D}_P=\overline{\textit{Ran}}D_P$.
Later Pal in \cite{PalNagy} (see also Theorem 4.4 of \cite{A:Pal}) proved that for any $\Gamma_{n}$-contraction $(S_1, \dots, S_{n-1},P)$, there is a unique operator tuple $(A_1, \dots, A_{n-1})$ on $\mathcal{D}_P$ such that 
\[
S_i-S_{n-i}^*P=D_PA_iD_P, \text{ for each }i= 1, \dots, n-1.
\] 
It is known (\cite{T:B}) that to for every $\mathbb E$-contraction $(A,B,P)$, there are unique operators $F_1, F_2\in \mathcal{B}(\mathcal{D}_P)$ such that 
\[
A-B^*P=D_PF_1D_P, \qquad B-A^*P=D_PF_2D_P.
\] 
The fundamental operator tuple for a $\Gamma_{n}$-contraction or the fundamental operator pair for an $\mathbb E$-contraction plays the central role in almost all results in the existing operator theoretic literature on these two domains. For this reason such a tuple or a pair was named the fundamental operator tuple or the fundamental operator pair respectively. \\

A contraction $P$ is called completely non-unitary (c.n.u) if it has no reducing non-trivial proper subspace on which its restriction is unitary. A $\Gamma_{n}$-contraction $(S_1, \dots,S_{n-1},P)$ is called c.n.u if $P$ is c.n.u. Sz.-Nagy-Foias constructed an explicit minimal isometric dilation (see \cite{Nagy}, CH-VI) for a c.n.u contraction.
\begin{defn}
	A commuting $n$-tuple of operators $\underline{T}=(T_1, \dots, T_{n})$ on a Hilbert space $\mathcal{H}$, having $X$ as a spectral set, is said to have a rational dilation if there exist a Hilbert space $\mathcal{K}$, an isometry $V:\mathcal{H}\to \mathcal{K}$ and an $n$-tuple of commuting normal operators $\underline{N}=(N_1, \dots,N_n)$ on $\mathcal{K}$ with $\sigma(\underline{N})\subseteq bX$ such that $f(\underline{T})=V^*f(\underline{N})V$, for every rational function on $X$, where $bX$ is the \textit{distinguished boundary} (to be defined in Section \ref{background}) of $X$.
\end{defn}

\begin{defn}
	Let $(S_1, \dots,S_{n-1},P)$ be a $\Gamma_{n}$-contraction on $\mathcal{H}$. A commuting tuple $(V_1,\dots, V_{n-1},V)$ on $\mathcal{K}$ is said to be a $\Gamma_{n}$-isometric dilation of $(S_1, \dots, S_{n-1},P)$ if $\mathcal{H}\subseteq\mathcal{K}$, $(V_1, \dots, V_{n-1},V)$ is a $\Gamma_{n}$-isometry and 
	\[
	P_{\mathcal{H}}\left(V_1^{m_1}\dots V_{n-1}^{m_{n-1}}V^m\right)|_{\mathcal{H}}=S_1^{m_1}\dots S_{n-1}^{m_{n-1}}P^m,
	\]
	for all non-negative integers $m_i, m$. Moreover, the dilation is called minimal if $\mathcal{K}=\overline{\text{span}}\left\{V_1^{m_1}\dots V_{n-1}^{m_{n-1}}V^mh:h \in \mathcal{H} \text{ and }m_i, m\in \mathbb N\cup \{0\}\right\}$.
\end{defn}

It was established in \cite{Agler} that the rational dilation does hold for the symmetrized bidisc. In \cite{SPrational}, a Schaffer type explicit dilation and functional model were obtained for a $\Gamma_{n}$-contraction under certain conditions. The author of this paper and Pal constructed an explicit Sz.-Nagy-Foias type $\Gamma_{n}$-isometric dilation for a particular class of c.n.u $\Gamma_{n}$-contractions in \cite{B:P5}. In Section \ref{dilationforGamma} (see Theorem \ref{bisai}), we present a necessary condition for the existence of $\Gamma_{n}$-isometric dilation of a c.n.u $\Gamma_{n}$-contraction. In the dilation theory and the Theorem \ref{bisai}, the $\mathcal{F}_O$-tuples play a pivotal role.\\

We now know how important the role of $\mathcal{F}_O$-tuples is in the operator theory on the symmetrized polydisc $\mathbb G_n$. So it is important to know which pair of operator tuples qualify as the $\mathcal{F}_O$-tuples of a $\Gamma_{n}$-contraction. As a consequence of Theorem \ref{bisai}, we develop a necessary condition on the $\mathcal{F}_O$-tuples of a c.n.u $\Gamma_{n}$-contraction $(S_1, \dots, S_{n-1},P)$ and its adjoint $(S_1^*, \dots, S_{n-1}^*,P^*)$. Indeed, Theorem \ref{necessarytheorem} leaves a necessary condition on $(A_1, \dots, A_{n-1})$ and $(B_1, \dots, B_{n-1})$ for them to be the $\mathcal{F}_O$-tuples of $(S_1, \dots, S_{n-1},P)$ and $(S_1^*, \dots, S_{n-1}^*,P^*)$ respectively. Theorem \ref{necessarytheorem} of this article is a generalized version of Lemma 3.1 in \cite{B:P} which deals with pure $\Gamma_{n}$-contractions. A contraction $P$ is called \textit{pure} if $P^{*n}$ strongly converges to $0$ as $n$ tends to infinity. A $\Gamma_{n}$-contraction $(S_1, \dots, S_{n-1},P)$ is called pure if the last component, that is, $P$ is pure. A natural question arises: what about the converse of Theorem \ref{necessarytheorem} ? That is, given two tuples of operators $(A_1, \dots, A_{n-1})$ and $(B_1, \dots, B_{n-1})$ defined on some certain Hilbert spaces, does there exist a c.n.u $\Gamma_{n}$-contraction $(S_1, \dots, S_{n-1},P)$ such that $(A_1, \dots, A_{n-1})$ becomes the $\mathcal{F}_O$-tuple of $(S_1, \dots, S_{n-1},P)$ and $(B_1, \dots, B_{n-1})$ becomes the $\mathcal{F}_O$-tuple of $(S_1^*, \dots, S_{n-1}^*,P^*)$. We answer this question in Theorem \ref{sufficient}. Again, Theorem \ref{sufficient} can be regarded as a generalized version of Theorem 3.4 in \cite{B:P} which deals with pure contraction. Also our results generalize the existing similar results for $\Gamma$-contractions \cite{TirthaHari}.\\

Let $(F_1, F_2)$ and $(G_1,G_2)$ be two pairs of operators defined on two Hilbert spaces. All the results obtained for $\Gamma_{n}$-contractions are then applied to $\mathbb E$-contractions to decipher when there is a c.n.u $\mathbb E$-contraction $(A,B,P)$ such that $(F_1,F_2)$ becomes the $\mathcal{F}_O$-pair of $(A,B,P)$ and $(G_1, G_2)$ becomes the $\mathcal{F}_O$-pair of $(A^*,B^*,P^*)$. This is the content of Theorem \ref{tetrablock} and it is a generalization of Theorem 3 in \cite{TirthaHari}.\\

By virtue of the unitary map $U:H^2(\mathbb D) \to H^2(\mathbb T)$ defined by
\[
z^n\mapsto e^{int},
\]
the Hilbert spaces $H^2(\mathbb D)$ and $H^2(\mathbb T)$ are unitarily equivalent. Also, for a Hilbert space $\mathcal{E}$, $H^2(\mathcal{E})$ is unitarily equivalent to $H^2(\mathbb D)\otimes \mathcal{E}$. In the sequel, we shall identify these unitarily equivalent spaces and use them, without mention, interchangeably as per notational convenience.

\section{Background Material and Preparatory Results}\label{background}

This section contains a few known facts about c.n.u contractions, $\Gamma_{n}$-contractions and $\mathbb E$-contractions.

\subsection{Isometric dilations of a c.n.u contraction}

\noindent In this subsection we recall two minimal isometric dilation of a c.n.u contraction. Let $P$ on $\mathcal{H}$ be a c.n.u contraction.\\
\begin{enumerate}
	\item[(i)] If $P$ is a c.n.u contraction, then $\{P^nP^{*n}: n \geq 1\}$ is a non-increasing sequence of positive operators so that it converges strongly. Suppose $\mathcal{A}_* $ is the strong limit of $\{P^nP^{*n}:n \geq 1\}$. Then $P^n\mathcal{A}_* P^{*n}=\mathcal{A}_* $ for every $n\geq 1$. This defines an isometry 
	\begingroup
	\allowdisplaybreaks
	\begin{align*}
	V_*:& \overline{\textit{Ran}}(\mathcal{A}_*) \to \overline{\textit{Ran}}(\mathcal{A}_*)\\
	& \mathcal{A}_*^{1/2}x \mapsto \mathcal{A}_*^{1/2}P^*x.
	\end{align*}
	\endgroup
	An interested reader can see CH-3 of \cite{Kubrusly} for more details of the map $V_*$. Since $V_*$ is an isometry on $\overline{\textit{Ran}}(\mathcal{A}_*)$, let $U$ on $\mathcal{K}$ be the minimal unitary extension of $V_*$. Let $D_P=(I-P^*P)^{1/2}$, $D_{P^*}=(I-PP^*)^{1/2}$, $\mathcal{D}_P=\overline{\textit{Ran}}D_P$ and $\mathcal{D}_{P^*}=\overline{\textit{Ran}}D_{P^*}$. Define an isometry 
	\begingroup
	\allowdisplaybreaks
	\begin{align*}
	\varphi_1 :& \mathcal{H} \to H^2(\mathcal{D}_{P^*}) \oplus \mathcal{K}\\
	& h \mapsto \sum\limits_{n=0}^{\infty}z^nD_{P^*}P^{*n}h \oplus \mathcal{A}_*^{1/2}h.
	\end{align*}
	\endgroup
	It was proved in \cite{Douglas} that $\begin{pmatrix}
	M_z & 0\\
	0 & U^*
	\end{pmatrix}$ on $H^2(\mathcal{D}_{P^*})\oplus \mathcal{K}$ is a minimal isometric dilation of $\varphi_1 P\varphi_1 ^*$ and 
	\begin{equation}\label{eqn8}
	\varphi_1 P^*= \begin{pmatrix}
	M_z & 0\\
	0 & U^*
	\end{pmatrix}^*\varphi_1.
	\end{equation}
	\item[(ii)] The notion of \textit{characteristic function} of a contraction introduced by Sz.-Nagy and Foias \cite{Nagy}. For a contraction $P$ on $\mathcal{H}$, let $\Lambda_P=\{z \in \mathbb C: (I-zP^*) \text{ is invertible}\}$. For $z\in \Lambda_P$, the characteristic function of $P$ is defined as
	\[
	\Theta_P(z)= [-P + zD_{P^*}(I-zP^*)^{-1}D_P]|_{\mathcal{D}_P}.
	\]
	By virtue of the relation $PD_P=D_{P^*}P$ (Section I.3 of \cite{Nagy}), $\Theta_P(z)$ maps $\mathcal{D}_P$ into $\mathcal{D}_{P^*}$ for every $z\in \Lambda_P$. Clearly, $\Theta_P$ induces a multiplication operator $M_{\Theta_P}$ from $H^2(\mathcal{D}_P)$ into $H^2(\mathcal{D}_{P^*})$ by 
	\[
	M_{\Theta_P}g(z)=\Theta_P(z)g(z), \text{ for all }g \in H^2(\mathcal{D}_P) \text{ and }z\in \mathbb D.
	\]
	Consider 
	\[
	\Delta_P(t)=(I_{\mathcal{D}_P}-\Theta_P(e^{it})^*\Theta_P(e^{it}))^{1/2}
	\]
	for those $t$ at which $\Theta_P(e^{it})$ exists, on $L^2(\mathcal{D}_P)$ and the subspaces 
	\begingroup
	\allowdisplaybreaks
	\begin{align*}
	&\textbf{K}=H^2(\mathcal{D}_P) \oplus \overline{\Delta_P L^2(\mathcal{D}_P)},\\
	&\mathcal{Q}_P=\{\Theta_P f \oplus \Delta_P f: f \in H^2(\mathcal{D}_P)\},\\
	&\mathcal{H}_P=\textbf{K}\ominus \mathcal{Q}_P.
	\end{align*}
	\endgroup
	If $P$ be a c.n.u contraction defined on a Hilbert space $\mathcal{H}$, then (see CH-VI of \cite{Nagy}) there exist an isometry $\varphi_2: \mathcal{H}\to \textbf{K}$ with $\varphi_2\mathcal{H}=\mathcal{H}_P$ such that $\begin{pmatrix}
	M_z & 0\\
	0 & M_{e^{it}}
	\end{pmatrix}$ on $\textbf{K}$ is a minimal isometric dilation of $\varphi_2P\varphi_2^*$ and 
	\begin{equation}\label{eqn9}
	\varphi_2 P^*= \begin{pmatrix}
	M_z & 0\\
	0 & M_{e^{it}}
	\end{pmatrix}^*\varphi_2.
	\end{equation}
\end{enumerate}
	
	Before proceeding further, we recall a result from \cite{B:P6} which will be used in sequel.
	
	\begin{lem}[\cite{B:P6}, Lemma 3.1]\label{BisaiPallem}
	Let $E_i, K_i$ for $i=1,2$, be Hilbert spaces. Suppose $U$ is a unitary from $(H^2(\mathbb D)\otimes E_1) \oplus K_1$ to $(H^2(\mathbb D)\otimes E_2) \oplus K_2$ such that 
	\[
	(M_z \otimes I_{E_1}) \oplus W_1 = U((M_z \otimes I_{E_2}) \oplus W_2)U^*,
	\]
	where $W_1$ on $K_1$ and $W_2$ on $K_2$ are unitaries. Then $U$ is of the form $(I_{H^2(\mathbb D)}\otimes U_1)\oplus U_2$ for some unitaries $U_1:E_1 \to E_2$ and $U_2:K_1 \to K_2$.
\end{lem}

 Since the isometries $\varphi_1$ and $\varphi_2$ give two minimal isometric dilations of c.n.u contraction $P$ and minimal isometric dilations of a contraction is unique up to unitary equivalence, therefore, there is a unitary $\mathfrak{U}: H^2(\mathcal{D}_{P^*})\oplus\overline{\Delta_P L^2(\mathcal{D}_P)}\to H^2(\mathcal{D}_{P^*})\oplus \mathcal{K}$ such that $\mathfrak{U}\varphi_2=\varphi_1$ and 
 \begin{equation}\label{eqn10}
 	\mathfrak{U}\begin{pmatrix}
 	M_z & 0\\
 	0 & M_{e^{it}}
 	\end{pmatrix}^*= \begin{pmatrix}
 	M_z & 0\\
 	0 & U^*
 	\end{pmatrix}^*\mathfrak{U}.
 \end{equation}
 Therefore, by Lemma \ref{BisaiPallem}, $\mathfrak{U}$ has the block matrix form
\begin{equation}\label{eqnU}
\mathfrak{U}=\begin{pmatrix}
I_{H^2(\mathbb D)}\otimes V_1 & 0\\
0 & V_2
\end{pmatrix},
\end{equation}
for some unitaries $V_1\in \mathcal{B}(\mathcal{D}_{P^*})$ and $V_2\in \mathcal{B}(\overline{\Delta_P L^2(\mathcal{D}_P)},\mathcal{K})$. Equations \eqref{eqn8}, \eqref{eqn9}, \eqref{eqn10} and \eqref{eqnU} will be used very frequently in sequels.\\

\subsection{$\Gamma_n$-contractions, $\mathbb E$-contractions and their special classes}
We recall from literature a few definitions and facts about the operators associated with $\Gamma_{n}$ and $\overline{\mathbb E}$.
\begin{defn}
	A commuting $n$-tuple of Hilbert space operators $(S_1, \dots,
	 S_{n-1},P)$ for which $\Gamma_{n}$ is a spectral set is called a $\Gamma_{n}$-contraction. Similarly, a commuting triple of Hilbert space operators $(A,B,P)$ for which $\overline{\mathbb{E}}$ is a spectral set is called an $\mathbb E$-contraction.
\end{defn}
The sets $\Gamma_{n}$ and $\overline{\mathbb E}$ are not convex but polynomially convex. By an application of Oka-Weil theorem, we have the following characterizations for $\Gamma_{n}$-contractions and for $\mathbb E$-contractions.

\begin{lem}[\cite{S:Pdecom}, Theorem 2.4]\label{SPLem}
	A commuting tuple of bounded operators $(S_1, \dots,S_{n-1},P)$ is a $\Gamma_{n}$-contraction if and only if 
	\[
	\|f(S_1, \dots, S_{n-1},P)\|\leq \|f\|_{\infty, \Gamma_{n}}
	\]
	for any holomorphic polynomial $f$ in $n$-variables.
\end{lem}

\begin{lem}[\cite{T:B}, Lemma 3.3]\label{TBLem}
	A commuting triple of bounded operators $(A,B,P)$ is a tetrablock contraction if and only if \[
	\|f(A,B,P)\|\leq \|f\|_{\infty,\overline{\mathbb E}}
	\]
	for any holomorphic polynomials $f$ in three variables. 
\end{lem}
It is evident from the above results that the adjoint of a $\Gamma_{n}$-contraction or an $\mathbb E$-contraction is also a $\Gamma_{n}$-contraction or an $\mathbb E$-contraction and if $(S_1, \dots, S_{n-1},P)$ is a $\Gamma_{n}$-contraction or $(A,B,P)$ is an $\mathbb E$-contraction, then $P$ is a contraction.

Unitaries, isometries, c.n.u, pure etc. are special classes of contractions. There are analogous classes for $\Gamma_{n}$-contractions and $\mathbb E$-contractions in the literature. Before discussing them, we shall recollect a few facts from literature.

For a compact subset $X$ of $\mathbb C^n$, let $\mathcal{A}(X)$ be the algebra of functions continuous in $X$ and holomorphic in the interior of $X$. A \textit{boundary} of $X$ (with respect to $\mathcal{A}(X)$) is a closed subset of $X$ on which every function in $\mathcal{A}(X)$ attains its maximum modulus. It follows from the theory of uniform algebras that the intersection of all boundaries of $X$ is also a boundary of $X$ (with respect to $\mathcal{A}(X)$) (see Theorem 9.1 of \cite{wermer}). This particular smallest one is called the \textit{distinguished boundary} of $X$ and is denoted by $bX$. The distinguished boundary of $\Gamma_{n}$, denoted by $b\Gamma_{n}$, was determined in \cite{E:Z} to be the symmetrization of the $n$-torus, i.e.,
\[
b\Gamma_{n}=\pi_n(\mathbb T^n)=\left\{\left(\sum\limits_{1\leq i\leq n}z_i, \sum\limits_{1\leq i<j\leq n}z_iz_j, \dots, \prod_{i=1}^{n} z_i\right): |z_i|=1, i= 1, \dots, n\right\}.
\]
Also, we obtained from literature (see \cite{A:W:Y}) that the distinguished boundary of $\overline{\mathbb{E}}$ is the following set:
\[
b\mathbb E=\left\{(a,b,p)\in \overline{\mathbb E}: |p|=1\right\}.
\]
Now we are in a position to present the special classes of $\Gamma_{n}$-contractions.
\begin{defn}
	Let $S_1,\dots, S_{n-1},P$ be commuting operators on $\mathcal H$.
	Then $(S_1,\dots, S_{n-1},P)$ is called
	\begin{itemize}
		\item[(i)] a $\Gamma_n$-\textit{unitary} if $S_1,\dots, S_{n-1},P$
		are normal operators and the Taylor joint spectrum $\sigma_T
		(S_1,\dots, S_{n-1},P)$ is a subset of $b\Gamma_n$ ;
		
		\item[(ii)] a $\Gamma_n$-isometry if there exist a Hilbert space
		$\mathcal K \supseteq \mathcal H$ and a $\Gamma_n$-unitary
		$(T_1,\dots,T_{n-1},U)$ on $\mathcal K$ such that $\mathcal H$ is
		a joint invariant subspace of $T_1,\dots, T_{n-1},U$ and that
		$(T_1|_{\mathcal H},\dots, \\T_{n-1}|_{\mathcal H},U|_{\mathcal
			H})=(S_1,\dots, S_{n-1},P)$;
		
		\item[(iii)] a c.n.u $\Gamma_n$-contraction if $(S_1,\dots, S_{n-1},P)$ is a $\Gamma_n$-contraction and $P$ is a c.n.u contraction;
		
		\item[(iv)] a pure $\Gamma_{n}$-contraction if $(S_1, \dots, S_{n-1},P)$ is a $\Gamma_{n}$-contraction and $P$ is a pure contraction.
	\end{itemize}
	
\end{defn}

We obtain from literature the following
analogous classes of $\mathbb E$-contractions.

\begin{defn}
	Let $A,B,P$ be commuting operators on $\mathcal H$. Then $(A,B,P)$
	is called
	\begin{itemize}
		\item[(i)] an $\mathbb E$-\textit{unitary} if $A,B,P$ are normal
		operators and the Taylor joint spectrum $\sigma_T (A,B,\\P)$ is a
		subset of $b\mathbb{E}$ ;
		
		\item[(ii)] an $\mathbb E$-isometry if there exists a Hilbert space
		$\mathcal K \supseteq \mathcal H$ and an $\mathbb E$-unitary
		$(Q_1,Q_2,V)$ on $\mathcal K$ such that $\mathcal H$ is a joint
		invariant subspace of $Q_1,Q_2,V$ and that $(Q_1|_{\mathcal H},
		Q_2|_{\mathcal H},\\V|_{\mathcal H})=(A,B,P)$ ;
		
		\item[(iii)] a c.n.u $\mathbb E$-contraction if $(A,B,P)$ is an $\mathbb E$-contraction and $P$  is a c.n.u contraction;
		
		\item[(iv)] a pure $\mathbb E$-contraction if $(A,B,P)$ is an $\mathbb E$-contraction and $P$ is a pure contraction.
	\end{itemize}
	
\end{defn}

The theorem below provides a set of characterizations of a $\Gamma_n$-unitary.
\begin{thm}[\cite{SPrational}, Theorem 4.2]\label{gammauni}
	Let $S_1, \dots, S_{n-1},P$ be commuting operators on  a Hilbert space $\mathcal{H}$. Then the following are equivalent:
	\begin{itemize}
		\item[(i)] $(S_1, \dots, S_{n-1},P)$ is a $\Gamma_n$-unitary;
		\item[(ii)] there exist commuting unitary operators $U_1, \dots, U_n$ on $\mathcal{H}$ such that \[
		(S_1, \dots, S_{n-1},P)=\pi_n(U_1, \dots, U_n);
		\]
		\item[(iii)] $(S_1, \dots, S_{n-1},P)$ is a $\Gamma_n$-contraction and $P$ is an unitary;
		\item[(iv)] $P$ is unitary, $S_i=S_{n-i}^*P$ for each $i=1, \dots, n-1$ and $\bigg(\dfrac{n-1}{n}S_1, \dfrac{n-2}{n}S_2, \dots,\\ \dfrac{1}{n}S_{n-1}\bigg)$ is a $\Gamma_{n-1}$-contraction.
	\end{itemize}
\end{thm}

For a given tuple, how does one determine whether it is a $\Gamma_{n}$-isometry or not? The next theorem gives necessary and sufficient conditions.

 \begin{thm}[\cite{SPrational}, Theorem 4.4]\label{gammaiso}
	Let $S_1, \dots, S_{n-1},P$ be commuting operators on  a Hilbert space $\mathcal{H}$. Then the following are equivalent:
	\begin{itemize}
		\item[(i)] $(S_1, \dots, S_{n-1},P)$ is a $\Gamma_n$-isometry;
		\item[(ii)] $(S_1, \dots, S_{n-1},P)$ is a $\Gamma_n$-contraction and $P$ is an isometry;
		\item[(iii)] $P$ is isometry, $S_i=S_{n-i}^*P$ for each $i=1, \dots, n-1$ and $\Big(\dfrac{n-1}{n}S_1, \dfrac{n-2}{n}S_2, \dots,\\ \dfrac{1}{n}S_{n-1}\Big)$ is a $\Gamma_{n-1}$-contraction.
	\end{itemize}
\end{thm}
We conclude this section by stating analogues of Theorems \ref{gammauni} and
\ref{gammaiso} for $\mathbb E$-contractions.

\begin{thm}[{\cite{T:B}, Theorem 5.4}]\label{thm:tu}
	Let $\underline N = (N_1, N_2, N_3)$ be a commuting triple of
	bounded operators. Then the following are equivalent.
	
	\begin{enumerate}
		
		\item $\underline N$ is an $\mathbb E$-unitary,
		
		\item $N_3$ is a unitary and $\underline N$ is an $\mathbb
		E$-contraction,
		
		\item $N_3$ is a unitary, $N_2$ is a contraction and $N_1 = N_2^*
		N_3$.
	\end{enumerate}
\end{thm}

\begin{thm}[{\cite{T:B}, Theorem 5.7}] \label{thm:ti}
	
	Let $\underline V = (V_1, V_2, V_3)$ be a commuting triple of
	bounded operators. Then the following are equivalent.
	
	\begin{enumerate}
		
		\item $\underline V$ is an $\mathbb E$-isometry.
		
		\item $V_3$ is an isometry and $\underline V$ is an $\mathbb
		E$-contraction.
		
		\item $V_3$ is an isometry, $V_2$ is a contraction and $V_1=V_2^*
		V_3$.
		
	\end{enumerate}
\end{thm}

\section{Necessary condition of conditional dilation for c.n.u $\Gamma_n$-contractions}\label{dilationforGamma}
	
	Let $(S_1, \dots, S_{n-1},P)$ be a $\Gamma_n$-contraction on $\mathcal{H}$ with $(A_1, \dots, A_{n-1})$ and $(B_1, \dots, B_{n-1})$ being the $\mathcal{F}_O$-tuples of $(S_1, \dots, S_{n-1},P)$ and $(S_1^*, \dots, S_{n-1}^*,P^*)$ respectively. Suppose
	\[
	\Sigma_1(z)=\left(\dfrac{n-1}{n}(A_1+zA_{n-1}^*), \dfrac{n-2}{n}(A_2+zA_{n-2}^*), \dots, \dfrac{1}{n}(A_{n-1}+zA_{1}^*)\right)
	\]
	and
	\[
	\Sigma_2(z)=\left(\dfrac{n-1}{n}(B_1^*+zB_{n-1}), \dfrac{n-2}{n}(B_2^*+zB_{n-2}), \dots, \dfrac{1}{n}(B_{n-1}^*+zB_{1})\right).
	\]
	The following theorem provides a conditional dilation of a c.n.u $\Gamma_n$-contraction.
	\begin{thm}[\cite{B:P6}, Corollary 3.6]\label{BisaiPalDilation}
		Let $(S_1, \dots, S_{n-1},P)$ be a c.n.u $\Gamma_n$-contraction on a Hilbert space $\mathcal{H}$ with $(A_1, \dots, A_{n-1})$ and $(B_1, \dots, B_{n-1})$ being the $\mathcal{F}_O$-tuples of $(S_1, \dots, S_{n-1},\\P)$ and $(S_1^*, \dots, S_{n-1}^*,P^*)$ respectively. Let $\Sigma_1(z)$ and $\Sigma_2(z)$ be $\Gamma_{n-1}$-contractions for all $z\in \mathbb D$. Then there is an isometry $\varphi_{BS}$ from $\mathcal{H}$ into $H^2(\mathcal{D}_{P^*})\oplus\overline{\Delta_P L^2(\mathcal{D}_P)}$ such that for $i=1, \dots, n-1$
		\[
		\varphi_{BS}S_i^*=\left(M_{G_i^*+zG_{n-i}}\oplus\widetilde{S}_{i2}\right)^*\varphi_{BS} \text{ and }\varphi_{BS}P^*=\left(M_z\oplus M_{e^{it}}\right)^*\varphi_{BS},
		\]
		where $(M_{G_1^*+zG_{n-1}}, \dots, M_{G_{n-1}^*+zG_1},M_z)$ on $H^2(\mathcal{D}_{P^*})$ is a pure $\Gamma_n$-isometry and $(\widetilde{S}_{12},\dots,\\ \widetilde{S}_{(n-1)2},M_{e^{it}})$ on $\overline{\Delta_P L^2(\mathcal{D}_P)}$ is a $\Gamma_n$-unitary.
	\end{thm}

	The following result from \cite{S:B} will be useful.
	\begin{lem}[\cite{S:B}, Theorem 4.10]\label{lem1}
		If $(T_1, \dots, T_{n-1},M_z)$ on $H^2(\mathcal{E})$ is a $\Gamma_n$-isometry, then for each $i=1, \dots, n-1$  $$T_i=M_{Y_i+zY_{n-i}^*},$$ for all $z\in \mathbb D$ and for some $Y_i\in\mathcal{B}(\mathcal{E})$.
	\end{lem}
	In order to prove the main result (Theorem \ref{bisai}) of this section, we need the following lemma.
	\begin{lem}\label{lemmain}
		Let $P$ be a c.n.u contraction on $\mathcal{H}$. Suppose $(M_{X_1^*+zX_{n-1}}, \dots, M_{X_{n-1}^*+zX_1},M_z)$ on $H^2(\mathcal{D}_{P^*})$ is a $\Gamma_n$-isometry and $(R_1, \dots, R_{n-1},M_{e^{it}})$ is a $\Gamma_n$-unitary on $\overline{\Delta_P L^2(\mathcal{D}_P)}$. If for each $i=1, \dots, n-1,$
		\begin{equation}\label{eqn1}
		\begin{pmatrix}
		M_{X_i^*+zX_{n-i}} & 0\\
		0 & R_i
		\end{pmatrix}\mathcal{Q}_P \subseteq \mathcal{Q}_P,
		\end{equation}
		then there exists a $\Gamma_n$-isometry $(M_{Y_1+zY_{n-1}^*}, \dots, M_{Y_{n-1}+zY_1^*},M_z)$ on $H^2(\mathcal{D}_P)$ such that for each $i=1, \dots, n-1,$	  
		\[
		\begin{pmatrix}
		M_{X_i^*+zX_{n-i}} & 0\\
		0 & R_i
		\end{pmatrix}\begin{pmatrix}
		M_{\Theta_P} \\
		\Delta_P
		\end{pmatrix}=\begin{pmatrix}
		M_{\Theta_P} \\
		\Delta_P
		\end{pmatrix}M_{{Y_i+zY_{n-i}^*}}.
		\]
		\begin{proof}
			From Equation \eqref{eqn1}, we can define operator $T_i$ on $H^2(\mathcal{D}_P)$ by the following way: 
			\begin{equation}\label{eqn2}
			\begin{pmatrix}
			M_{X_i^*+zX_{n-i}} & 0\\
			0 & R_i
			\end{pmatrix}\begin{pmatrix}
			M_{\Theta_P}\\
			\Delta_P
			\end{pmatrix}=\begin{pmatrix}
			M_{\Theta_P}\\
			\Delta_P
			\end{pmatrix}T_i
			\end{equation}
			for each $i$.
			Since $\begin{pmatrix}
			M_{\Theta_P}\\
			\Delta_P
			\end{pmatrix}$ is an isometry on $H^2(\mathcal{D}_P)$, 
			\begin{equation}\label{eqn3}
			T_i=\begin{pmatrix}
			M_{\Theta_P}\\
			\Delta_P
			\end{pmatrix}^*\begin{pmatrix}
			M_{X_i^*+zX_{n-i}} & 0\\
			0 & R_i
			\end{pmatrix}\begin{pmatrix}
			M_{\Theta_P}\\
			\Delta_P
			\end{pmatrix}.
			\end{equation}
			\textbf{Claim.} $(T_1, \dots, T_{n-1},M_z)$ on $H^2(\mathcal{D}_P)$ is a $\Gamma_n$-isometry. \\
			
			\noindent\textit{Proof of Claim.} Using Equation \eqref{eqn3} and taking cue from the facts that $\mathcal{Q}_P$ is invariant under $\begin{pmatrix}
			M_{X_i^*+zX_{n-i}} & 0\\
			0 & R_i
			\end{pmatrix}$ for each $i$ and that $\begin{pmatrix}
			M_{\Theta_P}\\
			\Delta_P
			\end{pmatrix}$ is an isometry we have
			\begingroup
			\allowdisplaybreaks
			\begin{align*}
			&T_iT_j\\  =& \begin{pmatrix}
			M_{\Theta_P}\\
			\Delta_P
			\end{pmatrix}^*\begin{pmatrix}
			M_{X_i^*+zX_{n-i}} & 0\\
			0 & R_i
			\end{pmatrix}\begin{pmatrix}
			M_{\Theta_P}\\
			\Delta_P
			\end{pmatrix}\begin{pmatrix}
			M_{\Theta_P}\\
			\Delta_P
			\end{pmatrix}^*\begin{pmatrix}
			M_{X_j^*+zX_{n-j}} & 0\\
			0 & R_j
			\end{pmatrix}\begin{pmatrix}
			M_{\Theta_P}\\
			\Delta_P
			\end{pmatrix}\\
			= & \begin{pmatrix}
			M_{\Theta_P}\\
			\Delta_P
			\end{pmatrix}^*\begin{pmatrix}
			M_{X_i^*+zX_{n-i}} & 0\\
			0 & R_i
			\end{pmatrix}\begin{pmatrix}
			M_{X_j^*+zX_{n-j}} & 0\\
			0 & R_j
			\end{pmatrix}\begin{pmatrix}
			M_{\Theta_P}\\
			\Delta_P
			\end{pmatrix}\\
			= & \begin{pmatrix}
			M_{\Theta_P}\\
			\Delta_P
			\end{pmatrix}^*\begin{pmatrix}
			M_{X_j^*+zX_{n-j}} & 0\\
			0 & R_j
			\end{pmatrix}\begin{pmatrix}
			M_{X_i^*+zX_{n-i}} & 0\\
			0 & R_i
			\end{pmatrix}\begin{pmatrix}
			M_{\Theta_P}\\
			\Delta_P
			\end{pmatrix}\\
			(&\text{since } M_{X_i^*+zX_{n-i}}M_{X_j^*+zX_{n-j}}=M_{X_j^*+zX_{n-j}}M_{X_i^*+zX_{n-i}}\text{ and }R_iR_j=R_jR_i)\\
			=&\begin{pmatrix}
			M_{\Theta_P}\\
			\Delta_P
			\end{pmatrix}^*\begin{pmatrix}
			M_{X_j^*+zX_{n-j}} & 0\\
			0 & R_j
			\end{pmatrix}\begin{pmatrix}
			M_{\Theta_P}\\
			\Delta_P
			\end{pmatrix}\begin{pmatrix}
			M_{\Theta_P}\\
			\Delta_P
			\end{pmatrix}^*\begin{pmatrix}
			M_{X_i^*+zX_{n-i}} & 0\\
			0 & R_i
			\end{pmatrix}\begin{pmatrix}
			M_{\Theta_P}\\
			\Delta_P
			\end{pmatrix}\\
			=&T_jT_i.
			\end{align*}
			\endgroup
			Again from Equation \eqref{eqn2}, we have 
			\begingroup
			\allowdisplaybreaks
			\begin{align}
			M_{X_i^*+zX_{n-i}}M_{\Theta_P}&=M_{\Theta_P}T_i,\, \label{eqn4}\\
			R_i\Delta_P&=\Delta_PT_i. \label{eqn5}
			\end{align}
			\endgroup
			Now applying $M_z$ and $M_{e^{it}}|_{\overline{\Delta_P L^2(\mathcal{D}_P)}}$ on both sides of Equations \eqref{eqn4} and \eqref{eqn5} respectively, we get
			\begingroup
			\allowdisplaybreaks
			\begin{align*}
			M_{\Theta_P}M_zT_i&=M_{\Theta_P}T_iM_z\\
			\Delta_PM_zT_i&=\Delta_PT_iM_z,
			\end{align*}
			\endgroup
			that is, $\begin{pmatrix}
			M_{\Theta_P}\\
			\Delta_P
			\end{pmatrix}M_zT_i=\begin{pmatrix}
			M_{\Theta_P}\\
			\Delta_P
			\end{pmatrix}T_iM_z \text{ and consequently } M_zT_i=T_iM_z$. Therefore, $(T_1,\\ \dots, T_{n-1},M_z)$ on $H^2(\mathcal{D}_P)$ is a commuting tuple. Clearly, the tuple $$\underline{\textbf{M}}=\left(\begin{pmatrix}
			M_{X_1^*+zX_{n-1}} & 0\\
			0 & R_1
			\end{pmatrix}, \dots, \begin{pmatrix}
			M_{X_{n-1}^*+zX_{1}} & 0\\
			0 & R_{n-1}
			\end{pmatrix}, \begin{pmatrix}
			M_{z} & 0\\
			0 & M_{e^{it}}
			\end{pmatrix}\right)$$ being a direct sum of two $\Gamma_{n}$-contractions, is a $\Gamma_{n}$-contraction. Now using Equation \eqref{eqn2} and the fact that $\begin{pmatrix}
			M_{\Theta_P}\\
			\Delta_P
			\end{pmatrix}$ is an isometry, for any holomorphic polynomial $f$ in $n$-variables we have
			$$
			\|f(T_1, \dots, T_{n-1},M_z)\|
			=\left\|\begin{pmatrix}
			M_{\Theta_P}\\
			\Delta_P
			\end{pmatrix}^*f(\underline{\textbf{M}})\begin{pmatrix}
			M_{\Theta_P}\\
			\Delta_P
			\end{pmatrix}\right\|\leq\|f\|_{\infty,\Gamma_{n}}.
			$$
			Therefore, by Lemma \ref{SPLem}, $(T_1, \dots, T_{n-1},M_z)$ on $H^2(\mathcal{D}_P)$ is a $\Gamma_n$-contraction. Since $M_z$ is isometry and hence by Theorem \ref{gammaiso}, $(T_1, \dots, T_{n-1},M_z)$ is a $\Gamma_{n}$-isometry. This completes the proof of the claim.\\
			
			\noindent By Lemma \ref{lem1}, $T_i=M_{Y_i+zY_{n-i}^*}$ for some $Y_i\in \mathcal{B}(\mathcal{D}_P)$. This completes the proof.
		\end{proof}
	\end{lem}

	Combining the Lemma \ref{lemmain} and Theorem \ref{BisaiPalDilation} we have the following theorem which is the main result of this section.
	\begin{thm}\label{bisai}
		Let $(S_1, \dots, S_{n-1},P)$ be a c.n.u $\Gamma_n$-contraction on a Hilbert space $\mathcal{H}$ with $(A_1, \dots, A_{n-1})$ and $(B_1, \dots, B_{n-1})$ being the $\ft$-tuples of $(S_1, \dots, S_{n-1},P)$ and $(S_1^*,\dots,\\ S_{n-1}^*,P^*)$ respectively. Let $\Sigma_1(z)$ and $\Sigma_2(z)$ be $\Gamma_{n-1}$-contractions for all $z\in \mathbb D$. Then for each $i=1, \dots, n-1$, there exist $Y_i\in \mathcal{B}(\mathcal{D}_P)$ such that $(M_{Y_1+zY_{n-1}^*},\dots, M_{Y_{n-1}+zY_1^*},M_z)$ on $H^2(\mathcal{D}_P)$ is a $\Gamma_n$-isometry and a unitary $\widetilde{V}_2\in \mathcal{B}(\overline{\Delta_P L^2(\mathcal{D}_P)})$ such that
		\[
		\begin{pmatrix}
		M_{\widetilde{G}_i^*+z\widetilde{G}_{n-i}} & 0\\
		0 & \widetilde{V}_2^*\widetilde{S}_{i2}\widetilde{V}_2
		\end{pmatrix}\begin{pmatrix}
		M_{\Theta_P}\\
		\Delta_P
		\end{pmatrix}=\begin{pmatrix}
		M_{\Theta_P}\\
		\Delta_P
		\end{pmatrix}M_{Y_i+zY_{n-i}^*},
		\]
		where the tuple $(\widetilde{G}_1, \dots, \widetilde{G}_{n-1})$ is unitarily equivalent to $(B_1,\dots, B_{n-1})$.
	\end{thm}
	\begin{proof}
		By Theorem \ref{BisaiPalDilation}, for each $i=1, \dots, n-1$
		\begin{equation}\label{eqn6}
		\varphi_{BS}S_i^*=\begin{pmatrix}
		M_{G_i^*+zG_{n-i}} & 0\\
		0 & \widetilde{S}_{i2}
		\end{pmatrix}^*\varphi_{BS} 
		\end{equation}
		and
		\[
		\varphi_{BS}P^*=\begin{pmatrix}
		M_z & 0\\
		0 & M_{e^{it}}
		\end{pmatrix}^*\varphi_{BS}.
		\]
		Since $P$ is a c.n.u contraction and minimal isometric dilation of a contraction is unique, there exists a unitary $\Phi:H^2(\mathcal{D}_{P^*})\oplus \overline{\Delta_P L^2(\mathcal{D}_P)} \to H^2(\mathcal{D}_{P^*})\oplus \overline{\Delta_P L^2(\mathcal{D}_P)}$ such that $\Phi\varphi_2=\varphi_{BS}$ and $\Phi \begin{pmatrix}
		M_z & 0\\
		0 & M_{e^{it}}
		\end{pmatrix}^*=\begin{pmatrix}
		M_z & 0\\
		0 & M_{e^{it}}
		\end{pmatrix}^*\Phi$. Then by Lemma \ref{BisaiPallem}, $\Phi$ has the block matrix form $\Phi=\begin{pmatrix}
		I\otimes\widetilde{V}_1 & 0\\
		0 & \widetilde{V}_2
		\end{pmatrix}$, for some unitaries $\widetilde{V}_1\in\mathcal{B}(\mathcal{D}_{P^*})$, $\widetilde{V}_2\in \mathcal{B}(\overline{\Delta_P L^2(\mathcal{D}_P)})$. Taking cue from this fact and Equation \eqref{eqn6}, we have 
		\begingroup
		\allowdisplaybreaks
		\begin{align*}
		\varphi_2 S_i^* & = \begin{pmatrix}
		I\otimes \widetilde{V}_1^* & 0\\
		0 & \widetilde{V}_2^*
		\end{pmatrix}\begin{pmatrix}
		I\otimes G_i + M_z^* \otimes G_{n-i} & 0\\
		0 & \widetilde{S}_{i2}^*
		\end{pmatrix}\begin{pmatrix}
		I\otimes \widetilde{V}_1 & 0\\
		0 & \widetilde{V}_2
		\end{pmatrix}\varphi_2\\
		&=\begin{pmatrix}
		I \otimes \widetilde{V}_1^*G_i\widetilde{V}_1+M_z^*\otimes \widetilde{V}_1^*G_{n-i}\widetilde{V}_1 & 0\\
		0 & \widetilde{V}_2^*\widetilde{S}_{i2}^*\widetilde{V}_2
		\end{pmatrix}\varphi_2.
		\end{align*}
		\endgroup
		Set $\widetilde{G}_i=\widetilde{V}_1^*G_i\widetilde{V}_1$ and $S_{i2}=\widetilde{V}_2^*\widetilde{S}_{i2}\widetilde{V}_2$. Then 
		\begin{equation}\label{eqn7}
		\varphi_2 S_i^* = \begin{pmatrix}
		M_{\widetilde{G}_i^*+z\widetilde{G}_{n-i}} & 0\\
		0 & S_{i2}
		\end{pmatrix}^*\varphi_2.
		\end{equation}
		By Theorem \ref{BisaiPalDilation}, $(M_{G_1^*+zG_{n-1}}, \dots, M_{G_{n-1}^*+zG_1},M_z)$ on $H^2(\mathcal{D}_{P^*})$ is a pure $\Gamma_n$-isometry and $(\widetilde{S}_{12}, \dots, \widetilde{S}_{(n-1)2},M_{e^{it}})$ on $\overline{\Delta_P L^2(\mathcal{D}_P)}$ is a $\Gamma_n$-unitary. Since $\widetilde{V}_1$ is unitary, $(M_{\widetilde{G}_1^*+z\widetilde{G}_{n-1}}, \dots,\\ M_{\widetilde{G}_{n-1}^*+z\widetilde{G}_1},M_z)$ on $H^2(\mathcal{D}_{P^*})$ being unitarily equivalent to $(M_{G_1^*+zG_{n-1}}, \dots, M_{G_{n-1}^*+zG_1},M_z)$, is a pure $\Gamma_n$-isometry. Similarly, one can prove that $(S_{12}, \dots, S_{(n-1)2},M_{e^{it}})$ on $\overline{\Delta_P L^2(\mathcal{D}_P)}$ is a $\Gamma_n$-unitary. From Equation \eqref{eqn7}, it is clear that $\varphi_2\mathcal{H}=\mathcal{H}_P$ is invariant under $\begin{pmatrix}
		M_{\widetilde{G}_i^*+z\widetilde{G}_{n-i}} & 0\\
		0 & S_{i2}
		\end{pmatrix}^*$, that is, $\mathcal{Q}_P$ is invariant under $\begin{pmatrix}
		M_{\widetilde{G}_i^*+z\widetilde{G}_{n-i}} & 0\\
		0 & S_{i2}
		\end{pmatrix}$. Then by Lemma \ref{lemmain}, there exists a $\Gamma_n$-isometry $(M_{Y_1+zY_{n-1}^*}, \dots, M_{Y_{n-1}+zY_1^*},M_z)$ on $H^2(\mathcal{D}_P)$ such that $\begin{pmatrix}
		M_{\widetilde{G}_i^*+z\widetilde{G}_{n-i}} & 0\\
		0 & S_{i2}
		\end{pmatrix}\begin{pmatrix}
		M_{\Theta_P}\\
		\Delta_P
		\end{pmatrix}=\begin{pmatrix}
		M_{\Theta_P}\\
		\Delta_P
		\end{pmatrix}M_{Y_i+zY_{n-i}^*}.$ By using Equations \eqref{eqn7} and \eqref{eqn9} we have
		\begingroup
		\allowdisplaybreaks
		\begin{align*}
		&S_i^*-S_{n-i}P^*\\
		=&\varphi_2^*\begin{pmatrix}
		M_{\widetilde{G}_i^*+z\widetilde{G}_{n-i}} & 0\\
		0 & \widetilde{S}_{i2}
		\end{pmatrix}^*\varphi_2-\varphi_2^*\begin{pmatrix}
		M_{\widetilde{G}_{n-i}^*+z\widetilde{G}_{i}} & 0\\
		0 & \widetilde{S}_{(n-i)2}
		\end{pmatrix}\varphi_2\varphi_2^*\begin{pmatrix}
		M_z & 0\\
		0 & M_{e^{it}}
		\end{pmatrix}^*\varphi_2\\
		=&\varphi_2^*\begin{pmatrix}
		P_{\mathbb{C}}\otimes \widetilde{G}_i & 0\\
		0 & 0
		\end{pmatrix}\varphi_2\; (\text{since }\widetilde{S}_{i2}^*=M_{e^{it}}^*\widetilde{S}_{(n-1)2}=\widetilde{S}_{(n-1)2}M_{e^{it}}^*)\\
		=& \varphi_1\begin{pmatrix}
		P_{\mathbb C}\otimes V_1\widetilde{G}_iV_1^* & 0\\
		0 & 0
		\end{pmatrix}\varphi_1\\
		=& D_{P^*}(V_1\widetilde{G}_iV_1^*)D_{P^*}.
		\end{align*}
		\endgroup
		Since $\ft$-tuple of a $\Gamma_n$-contraction is unique, $(V_1\widetilde{G}_1V_1^*, \dots, V_1\widetilde{G}_{n-1}V_1^*)=(B_1, \dots, B_{n-1})$. Therefore, the tuple $(\widetilde{G}_1, \dots, \widetilde{G}_{n-1})$ is unitarily equivalent to the $\ft$-tuple $(B_1, \dots, B_{n-1})$ of $(S_1^*, \dots, S_{n-1}^*,P^*)$. This completes the proof.
	\end{proof}

\section{The $\ft$-tuples of c.n.u $\Gamma_n$-contractions: necessary condition and sufficient condition}
We start with a technical lemma that will be required to the proof of Theorem \ref{necessarytheorem}.
\begin{lem}\label{lem}
	Let $(S_1, \dots, S_{n-1},P)$ be a c.n.u $\Gamma_n$-contraction as in Theorem \ref{bisai} with $S_{j2}=M_{{{n-1} \choose {j}} + {{n-1} \choose {j-1}}e^{it}}$ in the representation \eqref{eqn7} of $S_j$ for each $j$. Then for $j =1, \dots, n-1$
	\[
	{{n-1} \choose {j-1}}\mathcal{A}_*+{{n-1}\choose {j}}\mathcal{A}_*P^*=\mathcal{A}_*S_{n-j}^*
	\]
	\qquad\qquad\qquad\qquad\qquad\qquad and
	\[
	{{n-1} \choose {j-1}}P\mathcal{A}_* + {{n-1} \choose {j}}\mathcal{A}_*=P\mathcal{A}_*S_{n-j}^*.
	\]
\end{lem}
\begin{proof}
	From Equations \eqref{eqn10} and \eqref{eqnU} we have $U^*=V_2M_{e^{it}}V_2^*$. Again using \eqref{eqn7}, \eqref{eqnU} and $\mathfrak{U}\varphi_2=\varphi_1$ we have 
	\begingroup
	\allowdisplaybreaks
	\begin{align*}
	\varphi_1 S_j^* & = \begin{pmatrix}
	M_{V_1} & 0\\
	0 & V_2
	\end{pmatrix}\begin{pmatrix}
	M_{\widetilde{G}_j^* + z \widetilde{G}_{n-j}} & 0\\
	0 & S_{j2}
	\end{pmatrix}^*\begin{pmatrix}
	M_{V_1} & 0\\
	0 & V_2
	\end{pmatrix}^*\varphi_1\\
	&= \begin{pmatrix}
	M_{V_1\widetilde{G}_j^*V_1^* + zV_{1}\widetilde{G}_{n-j}V_1^*} & 0\\
	0 & {{n-1}\choose {j}}+{{n-1}\choose {j-1}}V_2M_{e^{it}}V_2^*
	\end{pmatrix}^*\varphi_1\\
	& = \begin{pmatrix}
	M_{B_j^* + zB_{n-j}} & 0\\
	0 & {{n-1}\choose {j}}+{{n-1}\choose {j-1}}U^*
	\end{pmatrix}^*\varphi_1 \quad (\text{since }V_1\widetilde{G}_{j}V_1^*=B_j).
	\end{align*}
	\endgroup	
	That is, $S_j^*=\varphi_1^*\begin{pmatrix}
	M_{B_j^*+zB_{n-j}} & 0\\
	0 & {{n-1}\choose {j}} + {{n-j}\choose {j-1}}U^*
	\end{pmatrix}^*\varphi_1$. Similarly, we have $\varphi_1P^*=\begin{pmatrix}
	M_z & 0\\
	0 & U^*
	\end{pmatrix}^*\varphi_1$, that is, $P^*=\varphi_1^*\begin{pmatrix}
	M_z & 0\\
	0 & U^*
	\end{pmatrix}^*\varphi_1$. Clearly $\textit{Ran}(\varphi_1)$ is invariant under $\begin{pmatrix}
	M_z & 0\\
	0 & U^*
	\end{pmatrix}^*$ and $\begin{pmatrix}
	M_{B_j^*+zB_{n-j}} & 0\\
	0 & {{n-1}\choose {j}} + {{n-j}\choose {j-1}}U^*
	\end{pmatrix}^*$ for $j=1, \dots, n-1$. Since $\varphi_1$ is isometry we have
	$P^{*n}= \varphi_1^*\begin{pmatrix}
	M_z^n & 0\\
	0 & U^{*n}
	\end{pmatrix}^*\varphi_1.$
	Therefore, $P^nP^{*n}=\varphi_1^*\begin{pmatrix}
	M_z^nM_z^{*n} & 0\\
	0 & I
	\end{pmatrix}\varphi_1$ and hence $\mathcal{A}_*=\varphi_1^*\begin{pmatrix}
	0 & 0\\
	0 & I
	\end{pmatrix}\varphi_1$. Now
	\begingroup
	\allowdisplaybreaks
	\begin{align*}
	&{{n-1} \choose {j-1}}\mathcal{A}_*+{{n-1}\choose {j}}\mathcal{A}_*P^*-\mathcal{A}_*S_{n-j}^*\\
	=& {{n-1} \choose {j-1}}\varphi_1^*\begin{pmatrix}
	0 & 0\\
	0 & I
	\end{pmatrix}\varphi_1+{{n-1} \choose {j}}\varphi_1^*\begin{pmatrix}
	0 & 0\\
	0 & I
	\end{pmatrix}\varphi_1\varphi_1^*\begin{pmatrix}
	M_z^* & 0\\
	0 & U
	\end{pmatrix}\varphi_1\\
	& \qquad\qquad - \varphi_1^*\begin{pmatrix}
	0 & 0\\
	0 & I
	\end{pmatrix}\varphi_1\varphi_1^*\begin{pmatrix}
	M_{B_{n-j}^*+zB_{j}} & 0\\
	0 & {{n-1} \choose {j-1}} +{{n-1} \choose {j}}U^*
	\end{pmatrix}^*\varphi_1\\
	=&\varphi_1^*\begin{pmatrix}
	0 & 0\\
	0 & {{n-1} \choose {j-1}}+{{n-1} \choose {j}}U
	\end{pmatrix}\varphi_1 - \varphi_1^*\begin{pmatrix}
	0 & 0\\
	0 & {{n-1} \choose {j-1}}+{{n-1} \choose {j}}U
	\end{pmatrix}\varphi_1\\
	=& 0.
	\end{align*}
	\endgroup
	Therefore, ${{n-1} \choose {j-1}}\mathcal{A}_*+{{n-1}\choose {j}}\mathcal{A}_*P^*=\mathcal{A}_*S_{n-j}^*$. Since $P\mathcal{A}_*P^*=\mathcal{A}_*$, hence ${{n-1} \choose {j-1}}P\mathcal{A}_*+{{n-1}\choose {j}}\mathcal{A}_*=P\mathcal{A}_*S_{n-j}^*$. This completes the proof.
\end{proof}

Before proceeding further, we recall a necessary and sufficient condition under which a tuple of operator becomes the $\ft$-tuple of a $\Gamma_n$-contraction.
\begin{thm}[\cite{B:P}, Theorem 2.1]\label{BisaiPal1}
	A tuple of operators $(A_1, \dots, A_{n-1})$ defined on $\mathcal{D}_P$ is the $\ft$-tuple of a $\Gamma_n$-contraction $(S_1, \dots, S_{n-1},P)$ if and only if $(A_1, \dots, A_{n-1})$ satisfy the following operator equations in $X_1, \dots, X_{n-1}$:
	\[
	D_PS_i=X_iD_P+X_{n-i}^*D_PP, \qquad i= 1,\dots, n-1.
	\]
\end{thm}
The next two results provide relations between the $\ft$-tuples of $\Gamma_{n}$-contractions $(S_1, \dots,\\ S_{n-1},P)$ and $(S_1^*, \dots, S_{n-1}^*,P^*)$.

\begin{lem}[\cite{B:P}, Lemma 2.4]\label{BPlem1}
	Let $(S_1, \dots, S_{n-1},P)$ be a $\Gamma_n$-contraction on a Hilbert space $\mathcal H$ and $(A_1, \dots, A_{n-1})$ and $(B_1, \dots, B_{n-1})$ be respectively the $\ft$-tuples of $(S_1, \dots, S_{n-1},P)$ and $(S_1^*, \dots, S_{n-1}^*,P^*)$. Then 
	\[
	D_PA_i = (S_iD_P-D_{P^*}B_{n-i}P)|_{\mathcal{D}_P}, \quad \text{for }i=1, \dots, n-1.
	\]
\end{lem}
\begin{lem}[\cite{B:P}, Lemma 2.5]\label{BPlem}
	Let $(S_1, \dots, S_{n-1},P)$ be a $\Gamma_n$-contraction on a Hilbert space $\mathcal H$ and $(A_1, \dots, A_{n-1})$ and $(B_1, \dots, B_{n-1})$ be respectively the $\ft$-tuples of $(S_1, \dots, S_{n-1},P)$ and $(S_1^*, \dots, S_{n-1}^*,P^*)$. Then 
	\[
	PA_i = B_i^*P|_{\mathcal D_P}, \quad \text{for }i=1, \dots, n-1.
	\]
\end{lem}
The following result gives a relation between the $\ft$-tuples of $\Gamma_n$-contraction $(S_1, \dots,\\S_{n-1},P)$ and its adjoint $(S_1^*, \dots, S_{n-1}^*,P^*)$ and the characteristic function $\Theta_P$ of $P$.
\begin{lem}[\cite{B:P}, Lemma 3.1]\label{BisaiPal}
	Let $(A_1,\dots, A_{n-1})$ and $(B_1, \dots, B_{n-1})$ be the $\ft$-tuples of a $\Gamma_n$-contraction $(S_1, \dots, S_{n-1},P)$ and its adjoint $(S_1^*,\dots, S_{n-1}^*,P^*)$ respectively. Then for each $i = 1, \dots, n-1$
	\[
	(B_i^* + zB_{n-i})\Theta_P(z)=\Theta_P(z)(A_i + zA_{n-i}^*), \quad \text{for all }z \in \mathbb D.
	\]
\end{lem}

The next theorem provides a necessary condition on the $\ft$-tuples of a c.n.u $\Gamma_n$-contraction $(S_1,\dots,S_{n-1},P)$ and its adjoint $(S_1^*,\dots, S_{n-1}^*,P^*)$.
\begin{thm}\label{necessarytheorem}
	Let $(S_1, \dots, S_{n-1},P)$ be a c.n.u $\Gamma_n$-contraction on a Hilbert space $\mathcal{H}$ with $(A_1, \dots, A_{n-1})$ and $(B_1, \dots, B_{n-1})$ being the $\ft$-tuples of $(S_1, \dots, S_{n-1},P)$ and $(S_1^*, \dots,\\ S_{n-1}^*,P^*)$ such that $S_{j2}={{n-1}\choose {j}}I+M_{{{n-1}\choose {j-1}}e^{it}}$ in the representation \eqref{eqn7} of $S_{j}$. Let $\Sigma_1(z)$ and $\Sigma_2(z)$ be $\Gamma_{n-1}$-contractions for all $z\in \mathbb{D}$. Then 
	\begin{equation}\label{eqnnecessary}
		\begin{pmatrix}
		M_{B_j^* + z B_{n-j}} & 0\\
		0 & M_{{{n-1}\choose {j}}+{{n-1}\choose {j-1}}e^{it}}
		\end{pmatrix}\begin{pmatrix}
		M_{\Theta_P}\\
		\Delta_P
		\end{pmatrix}=\begin{pmatrix}
		M_{\Theta_P}\\
		\Delta_P
		\end{pmatrix}M_{A_j+zA_{n-j}^*}.
	\end{equation}
	Moreover, if $V_1$ is as in \eqref{eqnU}, then there exists a $\Gamma_n$-isometry $(M_{Y_1+zY_{n-1}^*}, \dots, M_{Y_{n-1}+zY_1^*},\\M_z)$ on $H^2(\mathcal{D}_P)$ such that 
	\begin{equation}\label{necessary1}
		\begin{pmatrix}
		M_{B_j^* + zB_{n-j}} & 0\\
		0 & M_{{{n-1}\choose {j}}+{{n-1}\choose {j-1}}e^{it}}
		\end{pmatrix}\begin{pmatrix}
		M_{V_1}M_{\Theta_P}\\
		\Delta_P
		\end{pmatrix}=\begin{pmatrix}
		M_{V_1}M_{\Theta_P}\\
		\Delta_P
		\end{pmatrix}M_{Y_j+zY_{n-j}^*}.
	\end{equation}
\end{thm}
\begin{proof}
	By Lemma \ref{BisaiPal}, we have for each $j=1, \dots, n-1$
	\begin{equation}\label{relation}
		M_{B_j^* + z B_{n-j}}M_{\Theta_P}=M_{\Theta_P}M_{A_j+zA_{n-j}^*}.
	\end{equation}
	Now we have to prove that
	\begin{equation}\label{relation1}
		M_{{{n-1}\choose {j}}+{{n-1}\choose {j-1}}e^{it}}\Delta_P=\Delta_P M_{A_j + z A_{n-j}^*}.
	\end{equation}
	Since $M_{{{n-1}\choose {j}}+{{n-1}\choose {j-1}}e^{it}}$ commutes with $\Delta_P$ and $\Delta_P$ is non-negative, it suffices to prove that
	\begin{equation}\label{relation2}
		\Delta_P^2 M_{{{n-1}\choose {j}}+{{n-1}\choose {j-1}}e^{it}}= \Delta_P^2 M_{A_j + z A_{n-j}^*}.
	\end{equation}
	
\noindent	Notice that (for simplification see the Appendix)
	\begingroup
	\allowdisplaybreaks
	\begin{align}\label{expression1}
		&\Delta_P(t)^2\left[{{n-1}\choose {j}}+{{n-1}\choose {j-1}}e^{it}\right]\notag\\
		=&\left[{{n-1}\choose{j}}D_P\mathcal{A}_*D_P+{{n-1}\choose{j-1}}D_PP\mathcal{A}_*D_P\right] \notag\\
		&+ \sum\limits_{m=1}^{\infty}e^{imt}\left[{{n-1}\choose {j-1}}D_P\mathcal{A}_*P^{*(m-1)}D_P+{{n-1}\choose{j}}D_P\mathcal{A}_*P^{*m}D_P\right]\notag\\
		&+\sum\limits_{m=-\infty}^{-1}e^{imt}\left[{{n-1}\choose{j}}D_PP^{-m}\mathcal{A}_*D_{P}+{{n-1}\choose {j-1}}D_PP^{-m+1}\mathcal{A}_*D_P\right].
	\end{align}
	\endgroup

\noindent	Again (see the Appendix)
	\begingroup
	\allowdisplaybreaks
	\begin{align}\label{expression2}
		& \Delta_P(t)^2(A_j+e^{it}A_{n-j}^*)\notag\\
		=& \left[D_P^2A_j-D_PS_jD_P+D_PD_{P^*}B_{n-j}P+D_PP\mathcal{A}_*S_{n-j}^*D_P\right]\notag\\
		&+e^{it}\left[A_{n-j}^*D_P^2+P^*B_j^*D_{P^*}D_{P}-D_PS_{n-j}^*D_P+D_P\mathcal{A}_*S_{n-j}^*D_P\right]\notag\\
		&+\sum\limits_{m=2}^{\infty}e^{imt}D_P\mathcal{A}_*P^{*(m-1)}S_{n-j}^*D_P + \sum\limits_{m=-\infty}^{-1}e^{imt}D_PP^{-m+1}\mathcal{A}_*S_{n-j}^*D_P.
	\end{align}
	\endgroup
	Suppose $\mathcal{I}_m$ and $\mathcal{J}_m$ denote the coefficients of $e^{imt}$ in the right hand side of Equations \eqref{expression1} and \eqref{expression2} respectively. Clearly $\mathcal{I}_m,\,\mathcal{J}_m\in \mathcal{B}(\mathcal{D}_P)$. Now 
	\begingroup
	\allowdisplaybreaks
	\begin{align*}
		\mathcal{J}_0D_P=&D_P^2A_jD_P-D_PS_jD_P^2+D_PD_{P^*}B_{n-j}PD_P+D_PP\mathcal{A}_*S_{n-j}^*D_P^2\\
		=&D_P(D_PA_jD_P)-D_{P}S_j(I-P^*P)+D_P(D_{P^*}B_{n-j}D_{P^*})P+D_PP\mathcal{A}_*S_{n-j}^*D_P^2\\
		=&D_P(S_j-S_{n-j}^*P)-D_PS_j(I-P^*P)+D_P(S_{n-j}^*-S_jP^*)P+D_PP\mathcal{A}_*S_{n-j}^*D_P^2\\
		=&D_PP\mathcal{A}_*S_{n-j}^*D_P^2\\
		=&{{n-1} \choose {j-1}}D_PP\mathcal{A}_*D_P^2 + {{n-1} \choose {j}}D_P\mathcal{A}_*D_P^2\quad (\text{by Lemma }\ref{lem})\\
		=&\mathcal{I}_0D_P.
	\end{align*}
	\endgroup
	Since $\mathcal{I}_0, \,\mathcal{J}_0\in\mathcal{B}(\mathcal{D}_P)$, therefore, $\mathcal{I}_0=\mathcal{J}_0$. \\
	Again
	\begingroup
	\allowdisplaybreaks
	\begin{align*}
		&D_P\mathcal{J}_1\\=&D_PA_{n-j}^*D_P^2+D_PP^*B_j^*D_{P^*}D_{P}-D_P^2S_{n-j}^*D_P+D_P^2\mathcal{A}_*S_{n-j}^*D_P\\
		=&(D_PA_{n-j}^*D_P)D_P+P^*(D_{P^*}B_j^*D_{P^*})D_P-(I-P^*P)S_{n-j}^*D_P
		 + D_P^2\mathcal{A}_*S_{n-j}^*D_P\\
		=&(S_{n-j}-S_j^*P)^*D_P+P^*(S_j^*-S_{n-j}P^*)^*D_P-S_{n-j}^*D_P+P^*PS_{n-j}^*D_P+D_P^2\mathcal{A}_*S_{n-j}^*D_P\\
		=&D_P^2\mathcal{A}_*S_{n-j}^*D_P\\
		=&{{n-1}\choose{j-1}}D_P^2\mathcal{A}_*D_P+{{n-1}\choose{j}}D_P^2\mathcal{A}_*P^*D_P\\
		=&D_P\mathcal{I}_1.
	\end{align*}
	\endgroup
	Therefore, if $T=\mathcal{J}_1-\mathcal{I}_1$, then $T:\mathcal{D}_P\to \mathcal{D}_P$ and $D_PT=0$. Now 
	\[
	\langle TD_Ph, D_Ph' \rangle=\langle D_PTh,h' \rangle=0 \quad \text{for all }h,h'\in \mathcal{H}.
	\]
	This implies that $T=0$ and thus $\mathcal{I}_1=\mathcal{J}_1$.\\
	
	For $m\geq 2$
	\begingroup
	\allowdisplaybreaks
	\begin{align*}
	\mathcal{I}_m=&{{n-1}\choose {j-1}}D_P\mathcal{A}_*P^{*(m-1)}D_P+{{n-1}\choose{j}}D_P\mathcal{A}_*P^{*m}D_P\\
	=&D_P\left[{{n-1}\choose {j-1}}\mathcal{A}_*+ {{n-1}\choose{j}}\mathcal{A}_*P^{*}\right]P^{*(m-1)}D_P\\
	=&D_P\mathcal{A}_*S_{n-j}^*P^{*(m-1)}D_P\quad (\text{by Lemma }\ref{lem})\\
	=&\mathcal{J}_m.
	\end{align*}
	\endgroup
	Further, for $m\leq -1$,
	\begingroup
	\allowdisplaybreaks
	\begin{align*}
		\mathcal{I}_m=&{{n-1}\choose{j}}D_PP^{-m}\mathcal{A}_*D_{P}+{{n-1}\choose {j-1}}D_PP^{-m+1}\mathcal{A}_*D_P\\
		=&D_PP^{-m}\left[{{n-1}\choose{j}}\mathcal{A}_*+{{n-1}\choose{j-1}}P\mathcal{A}_*\right]D_P\\
		=&D_PP^{-m+1}\mathcal{A}_*S_{n-j}^*D_P\quad (\text{by Lemma }\ref{lem})\\
		=&\mathcal{J}_m.
	\end{align*}
	\endgroup
	Therefore, $\mathcal{I}_m=\mathcal{J}_m$ for all $m$. Hence
	\[
	\Delta_P^2M_{{{n-1}\choose {j}}+{{n-1}\choose {j-1}}e^{it}}=\Delta_P^2M_{A_j+zA_{n-j}^*},
	\]
	which implies that 
	\begin{equation}\label{eqnmain1}
	M_{{{n-1}\choose {j}}+{{n-1}\choose {j-1}}e^{it}}\Delta_P=\Delta_PM_{A_j+zA_{n-j}^*}.
	\end{equation}
	Therefore, combining Equations \eqref{relation} and \eqref{eqnmain1} we obtain Equation \eqref{eqnnecessary}.\\
	
	It is obvious from Equation \eqref{eqn7} that $\textit{Ran}(\varphi_2)=\mathcal{H}_P=\mathcal{Q}_P^{\perp}$ is invariant under $$\begin{pmatrix}
	M_{\widetilde{G}_j^*+z\widetilde{G}_{n-j}} & 0\\
	0 & M_{{{n-1}\choose {j}}+{{n-1}\choose {j-1}}e^{it}}
	\end{pmatrix}^*.$$ Then by Lemma \ref{lemmain}, there exists a $\Gamma_n$-isometry $(M_{Y_1+zY_{n-1}^*}, \dots, M_{Y_{n-1}+zY_1^*},M_z)$ on $H^2(\mathcal{D}_P)$ such that for each $i=1, \dots, n-1$
	\[
	\begin{pmatrix}
	M_{\widetilde{G}_j^*+z\widetilde{G}_{n-j}} & 0\\
	0 & M_{{{n-1}\choose {j}}+{{n-1}\choose {j-1}}e^{it}}
	\end{pmatrix}\begin{pmatrix}
	M_{\Theta_P}\\
	\Delta_P
	\end{pmatrix}=\begin{pmatrix}
	M_{\Theta_P}\\
	\Delta_P
	\end{pmatrix}M_{Y_j+zY_{n-j}^*}.
	\] 
	From Theorem \ref{bisai}, we have $\widetilde{G}_i=V_1^*B_iV_1$, where $V_1$ is as in \eqref{eqnU}. Therefore,
	\[\begin{pmatrix}
		M_{V_1^*B_j^*V_1 + z V_1^*B_{n-j}V_1} & 0\\
		0 & M_{{{n-1}\choose {j}}+{{n-1}\choose {j-1}}e^{it}}
	\end{pmatrix}\begin{pmatrix}
	M_{\Theta_P}\\
	\Delta_P
	\end{pmatrix}=\begin{pmatrix}
	M_{\Theta_P}\\
	\Delta_P
	\end{pmatrix}M_{Y_j+zY_{n-j}^*},
	\]
	that is,
	\[
	\begin{pmatrix}
	M_{B_j^* + z B_{n-j}} & 0\\
	0 & M_{{{n-1}\choose {j}}+{{n-1}\choose {j-1}}e^{it}}
	\end{pmatrix}\begin{pmatrix}
	M_{V_1}M_{\Theta_P}\\
	\Delta_P
	\end{pmatrix}=\begin{pmatrix}
	M_{V_1}M_{\Theta_P}\\
	\Delta_P
	\end{pmatrix}M_{Y_j+zY_{n-j}^*}.
	\]
\end{proof}

A straight-forward corollary of Theorem \ref{necessarytheorem} is the following.
\begin{cor}\label{bappacor}
	Let $(S_1, \dots, S_{n-1},P)$ be a c.n.u $\Gamma_n$-contraction on a Hilbert space $\mathcal{H}$ with $(A_1,\dots, A_{n-1})$ and $(B_1, \dots, B_{n-1})$ being the $\ft$-tuples of $(S_1, \dots, S_{n-1},P)$ and $(S_1^*,\dots,\\S_{n-1}^*,P^*)$ such that 
	\[
	\varphi_2S_j^*=\begin{pmatrix}
	M_{V_1^*B_j^*V_1+zV_1^*B_{n-j}V_1}& 0\\
	0 & M_{{{n-1}\choose {j}}+{{n-1}\choose{j-1}}e^{it}}
	\end{pmatrix}^*\varphi_2,
	\]
	where $V_1$ is as in \eqref{eqnU}. Suppose $\Sigma_2(z)$ is $\Gamma_{n-1}$-contractions for all $z\in \mathbb D$. Then Equation \eqref{eqnnecessary} holds.
\end{cor}
\begin{proof}
	One can prove it easily if follows the proof of Theorem \ref{necessarytheorem}.
\end{proof}

The Theorem \ref{necessarytheorem} provides a necessary condition on $(A_1, \dots, A_{n-1})$ and $(B_1, \dots, B_{n-1})$ to be the $\ft$-tuples of $(S_1, \dots, S_{n-1},P)$ and $(S_1^*,\dots, S_{n-1}^*,P^*)$, respectively. It is natural to ask about sufficiency, that is, for given two tuples of operators $(A_1, \dots, A_{n-1})$ and $(B_1, \dots, B_{n-1})$ defined on two Hilbert spaces, under what conditions there exists a c.n.u $\Gamma_n$-contraction $(S_1, \dots, S_{n-1},P)$ such that $(A_1, \dots, A_{n-1})$ becomes the $\ft$-tuple of $(S_1, \dots, S_{n-1},P)$ and $(B_1, \dots, B_{n-1})$ becomes the $\ft$-tuple of $(S_1^*, \dots, S_{n-1}^*,P^*)$.
\begin{thm}\label{sufficient}
	Let $P$ be a c.n.u contraction on a Hilbert space $\mathcal{H}$. Let $A_1, \dots, A_{n-1}\in \mathcal{B}(\mathcal{D}_P)$ and $B_1, \dots, B_{n-1}\in \mathcal{B}(\mathcal{D}_{P^*})$ be such that they satisfy Equations $\eqref{eqnnecessary}$ and \eqref{necessary1}. Suppose $\Sigma_2(z)$ is $\Gamma_{n-1}$-contraction for all $z\in \mathbb D$. Then there exists a $\Gamma_n$-contraction $(S_1, \dots, S_{n-1},P)$ such that $(A_1, \dots, A_{n-1})$ becomes the $\ft$-tuple of $(S_1, \dots, S_{n-1},P)$ and $(B_1, \dots, B_{n-1})$ becomes the $\ft$-tuple of $(S_1^*, \dots, S_{n-1}^*,P^*)$.
\end{thm}
\begin{proof}
	Let us define 
	\[
	S_j=\varphi_2^*W_j\varphi_2 \quad \text{for }j=1, \dots, n-1,
	\]
	 where 
	 \[
	 W_j=\begin{pmatrix}
	M_{V_1^*B_j^*V_1+zV_1^*B_{n-j}V_1}& 0\\
	0 & M_{{{n-1}\choose {j}}+{{n-1}\choose{j-1}}e^{it}}
	\end{pmatrix}\quad \text{for }j=1, \dots, n-1.
	\]
	 Equation \eqref{necessary1} tells us that $\mathcal{Q}_P \left(=(\textit{Ran}(\varphi_2))^{\perp}\right)$ is invariant under $W_j$ for $j=1, \dots, n-1$, that is, $\textit{Ran}(\varphi_2)$ is invariant under $W_j^*$ for $j=1, \dots, n-1$. Also from Equation \eqref{eqn9}, $P=\varphi_2^*W\varphi_2$ and $\textit{Ran}(\varphi_2)$ is invariant under $W^*$, where $W=\begin{pmatrix}
	 M_z & 0\\
	 0 & M_{e^{it}}
	 \end{pmatrix}$. Since $\Sigma_2(z)$ is $\Gamma_{n-1}$-contraction for all $z\in \mathbb D$, by Theorem \ref{gammaiso}, $\left(M_{B_1^*+zB_{n-1}}, \dots, M_{B_{n-1}^*+zB_1},M_z\right)$ on $H^2(\mathcal{D}_{P^*})$ is a $\Gamma_n$-isometry. Therefore, $\left(M_{V_1^*B_1^*V_1+zV_1^*B_{n-1}V_1}, \dots, M_{V_1^*B_{n-1}^*V_1+zV_1^*B_{1}V_1},M_z\right)$ $(=\Omega(z), \text{ say})$ is a $\Gamma_n$-isometry on $H^2(\mathcal{D}_{P^*})$ as $V_1$ on $\mathcal{D}_{P^*}$ is a unitary. Again by Theorem \ref{gammauni}, 
	 \begingroup
	 \allowdisplaybreaks\begin{align*}
	 	&\pi_n(I,\dots,I,\dots, I,M_{e^{it}})\\
	 	=&\left(M_{{{n-1}\choose {1}}+e^{it}}, \dots,M_{{{n-1}\choose {j}}+{{n-1}\choose{j-1}}e^{it}}, \dots, M_{I+{{n-1}\choose{n-2}}e^{it}},M_{e^{it}} \right)
	 \end{align*}
	 \endgroup
	 on $\overline{\Delta_P L^2(\mathcal{D}_P)}$ is a $\Gamma_n$-unitary. Hence the tuple $(W_1, \dots, W_{n-1},W)$ is a $\Gamma_{n}$-contraction.\\
	 
	 \noindent \textbf{Claim.} $\left(S_1, \dots, S_{n-1},P\right)$ is a $\Gamma_n$-contraction on $\mathcal{H}$.\\
	 
	 \noindent\textit{Proof of Claim.} Clearly, $\left(S_1^*, \dots, S_{n-1}^*,P^*\right)$ is a commuting tuple of operators and for any $f\in \mathbb C[z_1, \dots, z_{n}]$, 
	 \[
	 f(S_1^*,\dots, S_{n-1}^*,P^*)=\varphi_2^*f(W_1^*, \dots, W_{n-1}^*,W^*)\varphi_2.
	 \]
	 Since $(W_1^*,\dots, W_{n-1}^*,W^*)$ is a $\Gamma_n$-contraction, so  
	 \[
	 	\|f(S_1^*, \dots, S_{n-1}^*,P^*)\|=\left\|\varphi_2^*f(W_1^*, \dots, W_{n-1}^*,W^*)\varphi_2\right\|
	 	\leq\|f\|_{\infty, \Gamma_n}.
	 \]
	 Therefore, by Lemma \ref{SPLem}, $(S_1^*, \dots, S_{n-1}^*,P^*)$ is a $\Gamma_n$-contraction on $\mathcal{H}$ and again by Lemma \ref{SPLem}, it follows that $(S_1, \dots, S_{n-1},P)$ is a $\Gamma_{n}$-contraction as well. This completes the proof of the claim.\\
	 
	 We now show that $(B_1, \dots, B_{n-1})$ is the $\ft$-tuple of $(S_1^*, \dots, S_{n-1}^*,P^*)$. For each $j=1, \dots, n-1$
	 \begingroup
	 \allowdisplaybreaks
	 \begin{align*}
	 	S_j^*-S_{n-j}P^*=&\varphi_2^*W_j^*\varphi_2-\varphi_2^*W_{n-j}\varphi_2\varphi_2^*W^*\varphi_2\\
	 	=&\varphi_2^*W_j^*\varphi_2-\varphi_2^*W_{n-j}W^*\varphi_2\\
	 	=&\varphi_2^*\begin{pmatrix}
	 	P_{\mathbb C}\otimes V_1^*B_jV_1 & 0\\
	 	0 & 0
	 	\end{pmatrix}\varphi_2\\
	 	=&\varphi_1^*\begin{pmatrix}
	 	P_{\mathbb C}\otimes B_j & 0\\
	 	0 & 0
	 	\end{pmatrix}\varphi_1\\
	 	=&D_{P^*}B_jD_{P^*},
	 \end{align*}
	 \endgroup
	 where $P_{\mathbb C}$ is the orthogonal projection from $H^2(\mathbb D)$ onto the subspace consisting of constant functions in $H^2(\mathbb D)$. Since $\ft$-tuple of a $\Gamma_n$-contraction is unique, therefore, $(B_1, \dots, B_{n-1})$ is the $\ft$-tuple of $(S_1^*, \dots, S_{n-1}^*,P^*)$. \\
	 
	 Suppose $(Y_1, \dots, Y_{n-1})$ is the $\ft$-tuple of $(S_1, \dots, S_{n-1},P)$. Then by Corollary \ref{bappacor}, we have 
	 \begin{equation}\label{eqnnecessary1}
	 \begin{pmatrix}
	 M_{B_j^* + z B_{n-j}} & 0\\
	 0 & M_{{{n-1}\choose {j}}+{{n-1}\choose {j-1}}e^{it}}
	 \end{pmatrix}\begin{pmatrix}
	 M_{\Theta_P}\\
	 \Delta_P
	 \end{pmatrix}=\begin{pmatrix}
	 M_{\Theta_P}\\
	 \Delta_P
	 \end{pmatrix}M_{Y_j+zY_{n-j}^*}.
	 \end{equation}
	 Then from Equation \eqref{eqnnecessary} and Equation \eqref{eqnnecessary1}, we have for each $j=1, \dots, n-1$ that
	 \begin{equation}\label{eqnnecessary2}
	 	\begin{pmatrix}
	 	M_{\Theta_P}\\
	 	\Delta_P
	 	\end{pmatrix}M_{A_j+zA_{n-j}^*}=\begin{pmatrix}
	 	M_{\Theta_P}\\
	 	\Delta_P
	 	\end{pmatrix}M_{Y_j+zY_{n-j}^*}.
	 \end{equation}
	 Now from Equation \eqref{eqnnecessary2} and the fact that $\begin{pmatrix}
	 M_{\Theta_P}\\
	 \Delta_P
	 \end{pmatrix}$ is an isometry we have that
	  $$A_j+zA_{n-j}^*=Y_j+zY_{n-j}^* \quad \text{for all }j=1, \dots, n-1$$
	  and for all $z \in \mathbb D$. Therefore, $Y_j=A_j$ for all $j$ and hence $(A_1, \dots, A_{n-1})$ is the $\ft$-tuple of $(S_1, \dots, S_{n-1},P)$. The proof is complete.
\end{proof}

Combining Theorems \ref{necessarytheorem} \& \ref{sufficient}, we get the following theorem which is one of the main results of this paper.
\begin{thm}\label{mainthm}
	Let $(S_1, \dots, S_{n-1},P)$ be a c.n.u $\Gamma_n$-contraction on a Hilbert space $\mathcal{H}$ with $(A_1, \dots, A_{n-1})$ and $(B_1, \dots, B_{n-1})$ being the $\ft$-tuples of $(S_1, \dots, S_{n-1},P)$ and $(S_1^*, \dots,\\ S_{n-1}^*,P^*)$ such that $S_{j2}={{n-1}\choose {j}}I+M_{{{n-1}\choose {j-1}}e^{it}}$ in the representation \eqref{eqn7} of $S_{j}$. Suppose $\Sigma_1(z)$ and $\Sigma_2(z)$ are $\Gamma_{n-1}$-contractions for all $z\in \mathbb{D}$. Then 
	\begin{equation}\label{Eqnnecessary}
	\begin{pmatrix}
	M_{B_j^* + z B_{n-j}} & 0\\
	0 & M_{{{n-1}\choose {j}}+{{n-1}\choose {j-1}}e^{it}}
	\end{pmatrix}\begin{pmatrix}
	M_{\Theta_P}\\
	\Delta_P
	\end{pmatrix}=\begin{pmatrix}
	M_{\Theta_P}\\
	\Delta_P
	\end{pmatrix}M_{A_j+zA_{n-j}^*}.
	\end{equation}
	Moreover, if $V_1$ is as in \eqref{eqnU}, then there exists a $\Gamma_n$-isometry $(M_{Y_1+zY_{n-1}^*}, \dots, M_{Y_{n-1}+zY_1^*},\\M_z)$ on $H^2(\mathcal{D}_P)$ such that 
	\begin{equation}\label{Necessary1}
	\begin{pmatrix}
	M_{B_j^* + zB_{n-j}} & 0\\
	0 & M_{{{n-1}\choose {j}}+{{n-1}\choose {j-1}}e^{it}}
	\end{pmatrix}\begin{pmatrix}
	M_{V_1}M_{\Theta_P}\\
	\Delta_P
	\end{pmatrix}=\begin{pmatrix}
	M_{V_1}M_{\Theta_P}\\
	\Delta_P
	\end{pmatrix}M_{Y_j+zY_{n-j}^*}.
	\end{equation}
	
	Conversely, if $P$ is a c.n.u contraction on a Hilbert space $\mathcal{H}$ and $A_1, \dots, A_{n-1}\in \mathcal{B}(\mathcal{D}_P)$ and $B_1, \dots, B_{n-1}\in \mathcal{B}(\mathcal{D}_{P^*})$ are such that they satisfy Equations \eqref{Eqnnecessary} $\&$ \eqref{Necessary1} with $\Sigma_2(z)$ is $\Gamma_{n-1}$-contraction for all $z\in \mathbb D$, then there exists a $\Gamma_n$-contraction $(S_1, \dots, S_{n-1},P)$ such that $(A_1, \dots, A_{n-1})$ becomes the $\ft$-tuple of $(S_1, \dots, S_{n-1},P)$ and $(B_1, \dots, B_{n-1})$ becomes the $\ft$-tuple of $(S_1^*, \dots, S_{n-1}^*,P^*)$.
\end{thm}

\section{Results about $\mathbb E$-contractions}

This section is devoted to prove a theorem for c.n.u $\mathbb E$-contractions which is analogue to the Theorem \ref{mainthm} for c.n.u $\Gamma_n$-contractions.

Before proceeding further, we recall a few results from literature which will be useful to prove the main result of this section.
\begin{lem}[\cite{T:B}, Corollary 4.2]\label{Fopair1}
	The fundamental operators $F_1$ and $F_2$ of a tetrablock contraction $(A,B,P)$ are the unique bounded linear operators on $\mathcal{D}_P$ that satisfy the pair of operator equations
	\[
	D_PA=X_1D_P+X_2^*D_PP \quad\text{ and }\quad D_PB=X_2D_P+X_1^*D_PP.
	\]
\end{lem}

Here is an analogue of Lemma \ref{BisaiPal} in the tetrablock setting.

\begin{lem}[\cite{TirthaHari}, Lemma 17]\label{Hari}
	Let $F_1$ and $F_2$ be the fundamental operators of a tetrablock contraction $(A,B,P)$ and $G_1$ and $G_2$ be the fundamental operators of the tetrablock contraction $(A^*,B^*,P^*)$. Then for all $z\in \mathbb D$,
	\[
	(F_1^*+F_2z)\Theta_{P^*}(z)=\Theta_{P^*}(z)(G_1+G_2^*z)
	\]
	and 
	\[
	(F_2^*+F_1z)\Theta_{P^*}(z)=\Theta_{P^*}(z)(G_2+G_1^*z).
	\]
\end{lem}
\subsection{Necessary condition of conditional dilation for c.n.u $\Bbb E$-contractions}
\text{ }

\noindent We recall a few sentences (until Theorem 3.5) from \cite{B:P5} to recall a conditional $\mathbb E$-isometric dilation of a c.n.u $\mathbb E$-contraction $(A,B,P)$ on the minimal Sz.-Nagy and Foias isometric dilation space $\textbf{K}$ of the c.n.u contraction $P$.

Let $(A,B,P)$ be a c.n.u $\Bbb E$-contraction on a Hilbert space $\mathcal H$ with $\ft$-pair $(F_1,F_2)$. Let $[F_1,F_2]=0,\,[F_1,F_1^*]=[F_2,F_2^*]$. Then we have from Theorem 6.1 of \cite{T:B} that $(V_1, V_2, V_3)$ on $\mathcal{K}=\mathcal{H}\oplus \ell^2(\mathcal{D}_P)$ is the minimal $\mathbb E$-isometric dilation of $(A,B,P)$, where $V_3$ on $\mathcal{K}$ is the minimal isometric dilation of $P$. By Theorem 5.6 of \cite{T:B}, there is a decomposition of $\mathcal{K}$ into a direct sum $\mathcal{K}=\mathcal{K}_1 \oplus \mathcal{K}_2$ such that $\mathcal{K}_1$ and $\mathcal{K}_2$ reduce $V_1$, $V_2$, $V_3$ and $({V}_{11},{V}_{12},{V}_{13})=(V_1|_{\mathcal{K}_1}, V_2|_{\mathcal{K}_1},V_3|_{\mathcal{K}_1})$ on $\mathcal{K}_1$ is a pure $\mathbb E$-isometry while $({V}_{21},{V}_{22},{V}_{23})=(V_1|_{\mathcal{K}_2}, V_2|_{\mathcal{K}_2},V_3|_{\mathcal{K}_2})$ on $\mathcal{K}_2$ is an $\mathbb E$-unitary. Taking cue from this fact, we have the following conditional $\mathbb E$-isometric dilation for a c.n.u $\mathbb E$-contraction $(A,B,P)$.

\begin{thm}[\cite{B:P5}, Theorem 3.5]\label{commutativemodel}
	Let $(A,B,P)$ be a c.n.u $\mathbb{E}$-contraction on a Hilbert space $\mathcal{H}$ with $\ft$-pair $(F_1,F_2)$. Let $[F_1,F_2] = 0$, $[F_1,F_1^*] = [F_2,F_2^*]$. Suppose $(G_1,G_2)$ is the $\ft$-pair of $({V}_{11}^*,{V}_{12}^*,{V}_{13}^*)$. Then $(\widetilde{A}_1 \oplus \widetilde{A}_2,\widetilde{B}_1 \oplus \widetilde{B}_2,\widetilde{P}_1 \oplus \widetilde{P}_2)$ on $\textbf{K}$ is a minimal $\mathbb E$-isometric dilation of $(A,B,P)$,
	where 
	\begin{equation*}
	\left.\begin{aligned}
	\widetilde{A}_1&=I_{H^2(\mathbb D)}\otimes \mathcal V_1^* G_1^* \mathcal V_1\, +\, M_z\otimes \mathcal V_1^*G_2 \mathcal V_1\\
	\widetilde{B}_1&=I_{H^2(\mathbb D)}\otimes \mathcal V_1^* G_2^* \mathcal V_1 \,+\, M_z\otimes \mathcal V_1^* G_1 \mathcal V_1 \\
	\widetilde{P_1} &= M_z\otimes I_{\mathcal{D}_{P^*}}
	\end{aligned}
	\right\}
	\quad \text{on } \; \; H^2(\mathbb D)\otimes \mathcal{D}_{P^*}
	\end{equation*}
	and
	\[
	\widetilde{A}_2 = \mathcal V_2^*V_{21}\mathcal V_2\, , \, \widetilde{B}_2 = \mathcal V_2^*V_{22}\mathcal V_1\,, \widetilde{P_2}=M_{e^{it}}|_{\overline{\Delta_PL^2(\mathcal{D}_P)}} \quad \text{ on } \;\; \overline{\Delta_PL^2(\mathcal{D}_P)},
	\]
	for some unitaries $\mathcal V_1: \mathcal D_{P^*} \rightarrow  \mathcal D_{V_3^*}$ and $\mathcal V_2: \overline{\Delta_PL^2(\mathcal{D}_P)} \rightarrow \mathcal{K}_2$. If $\underline{S}= (S_1,S_2,S_3)$ is an $\mathbb E$-isometric dilation of $(A,B,P)$ such that $S_3$ is the minimal isometric dilation of $P$, then $\underline{S}$ is unitarily equivalent to $(\widetilde{A}_1 \oplus \widetilde{A}_2,\widetilde{B}_1 \oplus \widetilde{B}_2,\widetilde{P}_1 \oplus \widetilde{P}_2)$. Finally,
	\begingroup
	\allowdisplaybreaks
	\begin{align*}
	A &= P_{\mathcal{H}}(\widetilde{A}_1\oplus\widetilde{A}_2)|_{\mathcal{H}}\,,\\
	B & = P_{\mathcal{H}}(\widetilde{B}_1\oplus\widetilde{B}_2)|_{\mathcal{H}}\,,\\
	P & = P_{\mathcal{H}}((M_z\otimes I_{\mathcal{D}_{P^*}}) \oplus M_{e^{it}}|_{\overline{\Delta_PL^2(\mathcal{D}_P)}})|_{\mathcal{H}}\,.
	\end{align*}
	\endgroup
	As a direct consequence of the above theorem we have the following result.
\end{thm}
\begin{thm}\label{bisaipalEdilation}
	Let $(A,B,P)$ be a c.n.u $\Bbb E$-contraction on a Hilbert space $\mathcal H$ with $\ft$-pair $(F_1,F_2)$. Let $[F_1,F_2]=0,\,[F_1,F_1^*]=[F_2,F_2^*]$. Then there is an isometry $\varphi$ from $\mathcal{H}$ into $H^2(\mathcal{D}_{P^*}) \oplus \overline{\Delta_P L^2(\mathcal{D}_P)}$ such that 
	\[
	\varphi A^*= \left(M_{X_1^*+zX_2}\oplus\widetilde{A}_2\right)^*\varphi,\qquad \varphi B^*= \left(M_{X_2^*+zX_1}\oplus\widetilde{B}_2\right)^*\varphi
	\]
	\[
	\text{and }\;\; \varphi P^*=\left(M_z \oplus M_{e^{it}}\right)^*\varphi,
	\]
	where $\left(M_{X_1^*+zX_2}, M_{X_2^*+zX_1},M_z\right)$ on $H^2(\mathcal{D}_{P^*})$ is a pure $\Bbb E$-iosmetry and $\big(\widetilde{A}_2,\widetilde{B}_2,M_{e^{it}}\big)$ on $\overline{\Delta_P L^2(\mathcal{D}_P)}$ is an $\Bbb E$-unitary.
\end{thm}
\begin{proof}
	We apply the same argument as in the proof of Theorem \ref{commutativemodel} and the proof follows.
\end{proof}

Here is an analogue of Lemma \ref{lem1} for $\Bbb E$-contractions.
\begin{lem}\label{lem01}
	If $(A_1,A_{2},M_z)$ on $H^2(\mathcal{E})$ is an $\Bbb E$-isometry, then there exist $Y_i\in\mathcal{B}(\mathcal{E})$ such that $A_i=M_{Y_i+zY_{n-i}^*}$.
\end{lem}
\begin{proof}
	The proof is similar to the proof of Lemma \ref{lem1} and we skip it.
\end{proof}
We present an analogue of Lemma \ref{lemmain} for $\Bbb E$-contractions.
\begin{lem}\label{lemmain01}
	Let $P$ be a c.n.u contraction on a Hilbert space $\mathcal{H}$. Let $(M_{X_1^*+zX_{n-1}},M_{X_{2}^*+zX_1},\\M_z)$ on $H^2(\mathcal{D}_{P^*})$ be an $\Bbb E$-isometry and $(R_1, R_{2},M_{e^{it}})$ be an $\Bbb E$-unitary on $\overline{\Delta_P L^2(\mathcal{D}_P)}$. If for each $i=1,\,2,$
	\begin{equation}\label{eqn01}
	\begin{pmatrix}
	M_{X_i^*+zX_{3-i}} & 0\\
	0 & R_i
	\end{pmatrix}\begin{pmatrix}
	M_{\Theta_P} \\
	\Delta_P
	\end{pmatrix}H^2(\mathcal{D}_P)\subseteq \begin{pmatrix}
	M_{\Theta_P} \\
	\Delta_P
	\end{pmatrix}H^2(\mathcal{D}_P),
	\end{equation}
	then there exists an $\Bbb E$-isometry $(M_{Y_1+zY_{2}^*}, M_{Y_{2}+zY_1^*},M_z)$ on $H^2(\mathcal{D}_P)$ such that for each $i=1, 2,$	  
	\[
	\begin{pmatrix}
	M_{X_i^*+zX_{3-i}} & 0\\
	0 & R_i
	\end{pmatrix}\begin{pmatrix}
	M_{\Theta_P} \\
	\Delta_P
	\end{pmatrix}=\begin{pmatrix}
	M_{\Theta_P} \\
	\Delta_P
	\end{pmatrix}M_{{Y_i+zY_{3-i}^*}}.
	\]
\end{lem}
\begin{proof}
	Similar to that of the proof of Theorem 4.10 of \cite{S:B}.
\end{proof}

Combining Theorem \ref{bisaipalEdilation} and Lemma \ref{lemmain01} we have the following result.
\begin{thm}\label{bisai1}
	Let $(A,B,P)$ be a c.n.u $\Bbb E$-contraction on a Hilbert space $\mathcal{H}$ with $(F_1,F_2)$ and $(G_1,G_2)$ being the $\ft$-pairs of $(A,B,P)$ and $(A^*,B^*,P^*)$ respectively. Let $[F_1,F_2]=0,\,[F_1,F_1^*]=[F_2,F_2^*]$. Then there exists $Y_1,\,Y_2\in \mathcal{B}(\mathcal{D}_P)$ such that $(M_{Y_1+zY_{2}^*},M_{Y_{2}+zY_1^*},M_z)$ on $H^2(\mathcal{D}_P)$ is an $\Bbb E$-isometry and for $i=1,\,2$,
	\[
	\begin{pmatrix}
	M_{V_1^*G_i^*V_1+zV_1^*G_{3-i}V_1} & 0\\
	0 & \widetilde{V}_2^*R_i\widetilde{V}_2
	\end{pmatrix}\begin{pmatrix}
	M_{\Theta_P}\\
	\Delta_P
	\end{pmatrix}=\begin{pmatrix}
	M_{\Theta_P}\\
	\Delta_P
	\end{pmatrix}M_{Y_i+zY_{3-i}^*}.
	\]
\end{thm}
\begin{proof}
	We may imitate the proof of Theorem \ref{bisai}.
\end{proof}

\subsection{Admissible $\ft$-pairs}

We have an analogue of Lemma \ref{lem} for c.n.u $\Bbb E$-contractions.
\begin{lem}\label{lemEcontraction}
	Let $(A,B,P)$ be a c.n.u $\Bbb E$-contraction as in Theorem \ref{bisai1} with
	\begingroup
	\allowdisplaybreaks
	\begin{align*}
		&\varphi_2A^*=\begin{pmatrix}
		M_{V_1^*G_1^*V_1+zV_1^*G_{2}V_1} & 0\\
		0 & \dfrac{1}{2}(I+M_{e^{it}})
		\end{pmatrix}^*\varphi_2,\\
		& \varphi_2B^*=\begin{pmatrix}
		M_{V_1^*G_2^*V_1+zV_1^*G_{1}V_1} & 0\\
		0 & \dfrac{1}{2}(I+M_{e^{it}})
		\end{pmatrix}^*\varphi_2\\
		\text{and }\;\;& \varphi_2P^*=\begin{pmatrix}
		M_z & 0\\
		0 & M_{e^{it}}
		\end{pmatrix}^*\varphi_2.
	\end{align*}
	\endgroup
	Then \[
	\mathcal{A}_* + \mathcal{A}_*P^*=2\mathcal{A}_*A^*=2\mathcal{A}_*B^*
	\]
	and
	\[
	P\mathcal{A}_* + \mathcal{A}_* = 2P\mathcal{A}_*A^*=2P\mathcal{A}_*B^*.
	\]
	\begin{proof}
		Follows immediately by proceeding as in the proof of Lemma \ref{lem}.
	\end{proof}
\end{lem}

\begin{lem}
	The triple $\left(M_{(I+e^{it})/2},M_{(I+e^{it})/2},M_{e^{it}}\right)$ on $\overline{\Delta_P L^2(\mathcal{D}_P)}$ is an $\mathbb E$-unitary.
\end{lem}
\begin{proof}
	Directly follows from Theorem \ref{thm:tu} and we omit it.
\end{proof}

The following theorem gives a necessary condition on the $\ft$-pairs of a c.n.u $\Bbb E$-contraction $(A,B,P)$ and its adjoint $(A^*,B^*,P^*)$ which is the analogue of Theorem \ref{necessarytheorem} in the tetrablock setting.
\begin{thm}\label{necessarytheorem1}
	Let $(A,B,P)$ be a c.n.u $\Bbb E$-contraction on a Hilbert space $\mathcal{H}$ as in Lemma \ref{lemEcontraction}. Suppose $(F_1,F_2)$ is the $\ft$-pair of $(A,B,P)$ such that $[F_1,F_2]=0,\,[F_1,F_1^*]=[F_2,F_2^*]$. Then for $j=1,\,2$,
	\begin{equation}\label{eqnnecessaryE}
	\begin{pmatrix}
	M_{G_j^* + z G_{3-j}} & 0\\
	0 & M_{(I+e^{it})/2}
	\end{pmatrix}\begin{pmatrix}
	M_{\Theta_P}\\
	\Delta_P
	\end{pmatrix}=\begin{pmatrix}
	M_{\Theta_P}\\
	\Delta_P
	\end{pmatrix}M_{F_j+zF_{3-j}^*}.
	\end{equation}
	Moreover, if $V_1$ is as in \eqref{eqnU}, then there exists an $\Bbb E$-isometry $(M_{Y_1+zY_{2}^*},\\ M_{Y_{2}+zY_1^*},M_z)$ on $H^2(\mathcal{D}_P)$ such that 
	\begin{equation}\label{necessary1E}
	\begin{pmatrix}
	M_{G_j^* + z G_{3-j}} & 0\\
	0 & M_{(I+e^{it})/2}
	\end{pmatrix}\begin{pmatrix}
	M_{V_1}M_{\Theta_P}\\
	\Delta_P
	\end{pmatrix}=\begin{pmatrix}
	M_{V_1}M_{\Theta_P}\\
	\Delta_P
	\end{pmatrix}M_{Y_j+zY_{n-j}^*}.
	\end{equation}
\end{thm}
\begin{proof}
	One can prove it easily by using Lemmas \ref{Hari} \& \ref{lemEcontraction} and mimicking the proof of Theorem \ref{necessarytheorem}.
\end{proof}
The following result is an analogue of Theorem \ref{sufficient}.
\begin{thm}\label{sufficientE}
	Let $P$ be a c.n.u contraction on a Hilbert space $\mathcal{H}$. Let $F_1,  F_{2}\in \mathcal{B}(\mathcal{D}_P)$ and $G_1,G_{2}\in \mathcal{B}(\mathcal{D}_{P^*})$ be such that they satisfy Equations $\eqref{eqnnecessaryE}$ and \eqref{necessary1E}. Suppose $[F_1,F_2]=0,\,[F_1,F_1^*]=[F_2,F_2^*]$. Then there exists an $\Bbb E$-contraction $(A,B,P)$ such that $(F_1,F_2)$ becomes the $\ft$-pair of $(A,B,P)$ and $(G_1,G_2)$ becomes the $\ft$-pair of $(A^*,B^*,P^*)$.
\end{thm}
\begin{proof}
	We apply the same argument as in the proof of Theorem \ref{sufficient} and the proof follows.
\end{proof}

Combining Theorems \ref{necessarytheorem1} $\&$ \ref{sufficientE} we get the following theorem which is the main result of this section.
\begin{thm}\label{tetrablock}
	Let $(A,B,P)$ be a c.n.u $\Bbb E$-contraction on a Hilbert space $\mathcal{H}$ as in Lemma \ref{lemEcontraction}. Suppose $(F_1,F_2)$ is the $\ft$-pair of $(A,B,P)$ such that $[F_1,F_2]=0,\,[F_1,F_1^*]=[F_2,F_2^*]$. Then for $j=1,\,2$,
	\begin{equation}\label{eqnnecessaryEcont}
	\begin{pmatrix}
	M_{G_j^* + z G_{3-j}} & 0\\
	0 & M_{(I+e^{it})/2}
	\end{pmatrix}\begin{pmatrix}
	M_{\Theta_P}\\
	\Delta_P
	\end{pmatrix}=\begin{pmatrix}
	M_{\Theta_P}\\
	\Delta_P
	\end{pmatrix}M_{F_j+zF_{3-j}^*}.
	\end{equation}
	Moreover, if $V_1$ is as in \eqref{eqnU}, then there exists an $\Bbb E$-isometry $(M_{Y_1+zY_{2}^*},\\ M_{Y_{2}+zY_1^*},M_z)$ on $H^2(\mathcal{D}_P)$ such that 
	\begin{equation}\label{necessary1Econt}
	\begin{pmatrix}
	M_{G_j^* + z G_{3-j}} & 0\\
	0 & M_{(I+e^{it})/2}
	\end{pmatrix}\begin{pmatrix}
	M_{V_1}M_{\Theta_P}\\
	\Delta_P
	\end{pmatrix}=\begin{pmatrix}
	M_{V_1}M_{\Theta_P}\\
	\Delta_P
	\end{pmatrix}M_{Y_j+zY_{n-j}^*}.
	\end{equation}
	Conversely, if $P$ is a c.n.u contraction on a Hilbert space $\mathcal{H}$ and $F_1,  F_{2}\in \mathcal{B}(\mathcal{D}_P)$ and $G_1,G_{2}\in \mathcal{B}(\mathcal{D}_{P^*})$ are such that they satisfy Equations $\eqref{eqnnecessaryEcont}$ and \eqref{necessary1Econt} with $[F_1,F_2]=0,\,[F_1,F_1^*]=[F_2,F_2^*]$, then there exists an $\Bbb E$-contraction $(A,B,P)$ such that $(F_1,F_2)$ becomes the $\ft$-pair of $(A,B,P)$ and $(G_1,G_2)$ becomes the $\ft$-pair of $(A^*,B^*,P^*)$.
\end{thm}

\begin{center}
	\textbf{Acknowledgment}
\end{center}
The author acknowledges Harish-Chandra Research Institute, Prayagraj for warm hospitality. The author is grateful to Indian Statistical Institute, Kolkata and Prof. Debashish Goswami for a visiting scientist position under which most of the work was done. He wishes to thank the Stat-Math Unit, ISI Kolkata for all the facilities provided to him. The author has been supported by Prof. Debashish Goswami's J C Bose grant and the institute post-doctoral fellowship of HRI, Prayagraj.

\text{ }\\

\noindent \textbf{Appendix}\\

\noindent \textit{Proof of Equation \ref{expression1}:}
\begingroup
\allowdisplaybreaks
\begin{align*}
&\Delta_P(t)^2\left[{{n-1}\choose {j}}+{{n-1}\choose {j-1}}e^{it}\right]\\
=& \left[I-\Theta_P(e^{it})^*\Theta_P(e^{it})\right]\left[{{n-1}\choose {j}}+{{n-1}\choose {j-1}}e^{it}\right]\\
=&\left[{{n-1}\choose {j}}+{{n-1}\choose {j-1}}e^{it}\right]-\left[-P^*+\sum\limits_{m=0}^{\infty}e^{-i(m+1)t}D_PP^mD_{P^*}\right]\\
& \times \left[-P + \sum\limits_{m=0}^{\infty}e^{i(m+1)t}D_{P^*}P^{*m}D_P\right]\left[{{n-1}\choose j} + {{n-1}\choose {j-1}}e^{it}\right]\\
=&{{n-1}\choose j} + {{n-1}\choose {j-1}}e^{it} - \left[-P^*+\sum\limits_{m=-\infty}^{-1}e^{imt}D_PP^{-m-1}D_{P^*}\right]\\
& \times \Bigg[-{{n-1}\choose j}P + e^{it}\left({{n-1}\choose {j}}D_{P^*}D_P-{{n-1}\choose {j-1}}P\right)\\
&+\sum\limits_{m=2}^{\infty}e^{imt}D_{P^*}P^{*(m-2)}\textbf{E}\Bigg],\,\text{where }\textbf{E}=\left({{n-1}\choose {j}}P^* +{{n-1}\choose {j-1}}\right)D_P\\
=&\Bigg[{{n-1}\choose {j}}D_P^2-D_PD_{P^*}\left({{n-1}\choose {j}}D_{P^*}D_P-{{n-1}\choose {j-1}}P\right) -\sum\limits_{k=2}^{\infty}D_PP^{k-1}D_{P^*}^2P^{*(k-2)}\textbf{E}\Bigg]\\
+&\sum\limits_{m=1}^{\infty}e^{imt}\Bigg[D_{P}P^{*(m-1)}\textbf{E}-\sum\limits_{k=1}^{\infty}D_PP^{k-1}D_{P^*}^2P^{*(m+k-2)}\textbf{E}\Bigg] + \sum\limits_{-\infty}^{-1}e^{imt}\Bigg[{{n-1}\choose{j}}D_PP^{-m-1}D_{P^*}P\\
-&D_PP^{-m}D_{P^*}\left({{n-1}\choose {j}}D_{P^*}D_P-{{n-1}\choose{j-1}}P\right)
-\sum\limits_{k=-\infty}^{m-2}D_PP^{-k-1}D_{P^*}^2P^{*(m-k-2)}\textbf{E}\Bigg].
\end{align*}
\endgroup
The last equality follows from the fact that $P^*D_{P^*}=D_PP^*$. Now consider the coefficients of the above expression:\\
\textbf{Constant term:}
\begingroup
\allowdisplaybreaks
\begin{align*}
&{{n-1}\choose {j}}D_P^2-D_PD_{P^*}\left({{n-1}\choose {j}}D_{P^*}D_P-{{n-1}\choose {j-1}}P\right)-\sum\limits_{k=2}^{\infty}D_PP^{k-1}D_{P^*}^2P^{*(k-2)}\textbf{E}\\
=&{{n-1}\choose {j}}D_P^2-{{n-1}\choose {j}}D_P^2+{{n-1}\choose {j}}D_PPP^*D_{P}+{{n-1}\choose {j-1}}D_PD_{P^*}P\\
&-\sum\limits_{k=2}^{\infty}D_PP\left(P^{k-2}P^{*(k-2)}-P^{k-1}P^{-(k-1)}\right)\textbf{E}\\
=&{{n-1}\choose {j}}D_PPP^*D_{P}+{{n-1}\choose {j-1}}D_PD_{P^*}P-D_PP(I-\mathcal{A}_*)\textbf{E}\\
=& {{n-1}\choose {j}}D_PPP^*D_{P}+{{n-1}\choose {j-1}}D_PPD_{P}-{{n-1}\choose {j}}D_PPP^*D_P\\
&-{{n-1}\choose{j-1}}D_PPD_P+{{n-1}\choose{j}}D_PP\mathcal{A}_*P^*D_P+{{n-1}\choose{j-1}}D_PP\mathcal{A}_*D_P\\
=&{{n-1}\choose{j}}D_P\mathcal{A}_*D_P+{{n-1}\choose{j-1}}D_PP\mathcal{A}_*D_P \;\;\;\; (\text{since }P\mathcal{A}_*P^*=\mathcal{A}_*).
\end{align*}
\endgroup
The second last equality follows from the fact that $D_{P^*}P=PD_P$.\\

\noindent \textbf{Coefficient of $e^{imt}$, $m\geq 1$:}
\begingroup
\allowdisplaybreaks
\begin{align*}
&D_{P}P^{*(m-1)}\textbf{E}
-\sum\limits_{k=1}^{\infty}D_PP^{k-1}D_{P^*}^2P^{*(m+k-2)}\textbf{E}\\
=& {{n-1}\choose {j}}D_{P}P^{*m}D_P+{{n-1}\choose {j-1}}D_{P}P^{*(m-1)}D_P-\sum\limits_{k=1}^{\infty}D_P\left(P^{k-1}P^{*(k-1)}-P^{k}P^{*k}\right)P^{*(m-1)}\textbf{E}\\
=& {{n-1}\choose {j}}D_{P}P^{*m}D_P+{{n-1}\choose {j-1}}D_{P}P^{*(m-1)}D_P \\
&-D_P(I-\mathcal{A}_*)P^{*(m-1)}\left({{n-1}\choose{j-1}}+{{n-1}\choose{j}}P^*\right)D_P\\
=&{{n-1}\choose {j}}D_PP^{*m}D_P + {{n-1}\choose {j-1}}D_PP^{*(m-1)}D_P-{{n-1}\choose {j-1}}D_PP^{*(m-1)}D_P\\
&-{{n-1}\choose {j}}D_PP^{*m}D_P+{{n-1}\choose {j-1}}D_P\mathcal{A}_*P^{*(m-1)}D_P+{{n-1}\choose{j}}D_P\mathcal{A}_*P^{*m}D_P\\
=&{{n-1}\choose {j-1}}D_P\mathcal{A}_*P^{*(m-1)}D_P+{{n-1}\choose{j}}D_P\mathcal{A}_*P^{*m}D_P.
\end{align*}
\endgroup

\textbf{Coefficient of $e^{imt}$, $m\leq -1$:}
\begingroup
\allowdisplaybreaks
\begin{align*}
&{{n-1}\choose{j}}D_PP^{-m-1}D_{P^*}P
-D_PP^{-m}D_{P^*}\left({{n-1}\choose {j}}D_{P^*}D_P-{{n-1}\choose{j-1}}P\right)\\
&-\sum\limits_{k=-\infty}^{m-2}D_PP^{-k-1}D_{P^*}^2P^{*(m-k-2)}\textbf{E}\\
=&{{n-1}\choose{j}}D_PP^{-m}D_{P}-{{n-1}\choose{j}}D_PP^{-m}D_{P^*}^2D_P+{{n-1}\choose{j-1}}D_PP^{-m+1}D_P\\
&-\sum\limits_{k=2-m}^{\infty}D_PP^{-m+1}\left(P^{m+k-2}P^{*(m+k-2)}-P^{m+k-1}P^{*(m+k-1)}\right)\textbf{E}\\
=&{{n-1}\choose{j}}D_PP^{-m}D_{P}-{{n-1}\choose{j}}D_PP^{-m}D_{P}+{{n-1}\choose{j}}D_PP^{-m+1}P^*D_{P}\\
&+ {{n-1}\choose{j-1}}D_PP^{-m+1}D_{P}-D_PP^{-m+1}(I-\mathcal{A}_*)\textbf{E}\\
=&{{n-1}\choose{j}}D_PP^{-m+1}P^*D_{P}+{{n-1}\choose{j-1}}D_PP^{-m+1}D_{P}-{{n-1}\choose{j}}D_PP^{-m+1}P^*D_{P}\\
&-{{n-1}\choose{j-1}}D_PP^{-m+1}D_{P}+{{n-1}\choose{j}}D_PP^{-m+1}\mathcal{A}_*P^*D_{P}+{{n-1}\choose {j-1}}D_PP^{-m+1}\mathcal{A}_*D_P\\
=&{{n-1}\choose{j}}D_PP^{-m}\mathcal{A}_*D_{P}+{{n-1}\choose {j-1}}D_PP^{-m+1}\mathcal{A}_*D_P.
\end{align*}
\endgroup
The last equality follows from the fact that $P\mathcal{A}_*P^*=\mathcal{A}_*$. Therefore, Equation \eqref{expression1} holds.\\

\noindent \textit{Proof of Equation \eqref{expression2}:}
\begingroup
\allowdisplaybreaks
\begin{align*}
&\Delta_P(t)^2(A_j+e^{it}A_{n-j}^*)\\
=&(I-\Theta_P(e^{it})^*\Theta_P(e^{it}))(A_j+e^{it}A_{n-j}^*)\\
=&(A_j+e^{it}A_{n-j}^*)-\Theta_P(e^{it})^*(B_j^*+ e^{it}B_{n-j})\Theta_P(e^{it})\;\;(\text{using Lemma } \ref{BisaiPal})\\
=&(A_j+e^{it}A_{n-j}^*)-\left[-P^*+\sum\limits_{m=0}^{\infty}e^{-i(m+1)t}D_PP^mD_{P^*}\right][B_j^*+ e^{it}B_{n-j}]\\
& \qquad\qquad\qquad\times \left[-P+\sum\limits_{m=0}^{\infty}e^{i(m+1)}D_{P^*}P^{*m}D_P\right]\\
=&A_j+e^{it}A_{n-j}^* -\left[-P^*+\sum\limits_{m=0}^{\infty}e^{-i(m+1)t}D_PP^mD_{P^*}\right] \Bigg[-B_j^*P\\
&+e^{it}(B_j^*D_{P^*}D_P-B_{n-j}P)+\sum\limits_{m=2}^{\infty}e^{imt}(B_j^*D_{P^*}P^*+B_{n-j}D_{P^*})P^{*(m-2)}D_P\Bigg]\\
=&A_j+e^{it}A_{n-j}^* -\left[-P^*+\sum\limits_{m=-\infty}^{-1}e^{imt}D_PP^{-m-1}D_{P^*}\right] \Bigg[-B_j^*P+e^{it}(B_j^*D_{P^*}D_P-B_{n-j}P)+\\
&\sum\limits_{m=2}^{\infty}e^{imt}D_{P^*}S_{n-j}^*P^{*(m-2)}D_P\Bigg]\big(\text{using Theorem }\ref{BisaiPal1} \text{ for }(S_1^*, \dots, S_{n-1}^*,P^*)\big)\\
=& \Bigg[A_j-P^*B_j^*P - D_PD_{P^*}(B_j^*D_{P^*}D_P-B_{n-j}P)-\sum\limits_{k=-\infty}^{-2}D_PP^{-k-1}D_{P^*}^2P^{*(-k-2)}S_{n-j}^*D_P\Bigg]\\
&+e^{it}\Bigg[A_{n-j}^*-P^*B_{n-j}P+P^*B_j^*D_{P^*}D_P-\sum\limits_{k=1}^{\infty}D_PP^{k-1}D_{P^*}^2P^{*(k-1)}S_{n-j}^*D_P\Bigg]\\
&+\sum\limits_{m=2}^{\infty}e^{imt}\Bigg[P^*D_{P^*}S_{n-j}^*P^{*(m-2)}D_P-\sum\limits_{k=1}^{\infty}D_PP^{k-1}D_{P^*}^2P^{*(m+k-2)}S_{n-j}^*D_P\Bigg]\\
&+\sum\limits_{m=-\infty}^{-1}e^{imt}\Bigg[D_PP^{-m-1}D_{P^*}B_j^*P-D_PP^{-m}D_{P^*}(B_j^*D_{P^*}D_P-B_{n-j}P)\\
&-\sum\limits_{k=2-m}^{\infty}D_PP^{k-1}D_{P^*}^2P^{*(m+k-2)}S_{n-j}^*D_P\Bigg].
\end{align*}
\endgroup
Let us simplify the coefficients of the above expression:\\
\textbf{Constant term:}
\begingroup
\allowdisplaybreaks
\begin{align*}
&A_j-P^*B_j^*P - D_PD_{P^*}(B_j^*D_{P^*}D_P-B_{n-j}P)-\sum\limits_{k=-\infty}^{-2}D_PP^{-k-1}D_{P^*}^2P^{*(-k-2)}S_{n-j}^*D_P\\
=&A_j-P^*PA_j-D_P(S_j^*-S_{n-j}P^*)^*D_P+D_PD_{P^*}B_{n-j}P\\
&-\sum\limits_{k=2}^{\infty}D_{P}P\left(P^{k-2}P^{*(k-2)}-P^{k-1}P^{*(k-1)}\right)S_{n-j}^*D_P\;\;(\text{by Lemma }\ref{BPlem})\\
=&D_P^2A_j-D_PS_jD_P+D_PPS_{n-j}^*D_P+D_PD_{P^*}B_{n-j}P-D_PP(I-\mathcal{A}_*)S_{n-j}^*D_P\\
=&D_P^2A_j-D_PS_jD_P+D_PD_{P^*}B_{n-j}P+D_PP\mathcal{A}_*S_{n-j}^*D_P.
\end{align*}
\endgroup
\textbf{Coefficient of $e^{it}$:}
\begingroup
\allowdisplaybreaks
\begin{align*}
&A_{n-j}^*-P^*B_{n-j}P
+P^*B_j^*D_{P^*}D_P-\sum\limits_{k=1}^{\infty}D_PP^{k-1}D_{P^*}^2P^{*(k-1)}S_{n-j}^*D_P\\
=&A_{n-j}^*-A_{n-j}^*P^*P+P^*B_j^*D_{P^*}D_{P}-D_P(I-\mathcal{A}_*)S_{n-j}^*D_P\;(\text{by Lemma }\ref{BPlem})\\
=&A_{n-j}^*D_P^2+P^*B_j^*D_{P^*}D_{P}-D_PS_{n-j}^*D_P+D_P\mathcal{A}_*S_{n-j}^*D_P.
\end{align*}
\endgroup
\textbf{Coefficient of $e^{imt}$, $m\geq 2$:}
\begingroup
\allowdisplaybreaks
\begin{align*}
&P^*D_{P^*}S_{n-j}^*P^{*(m-2)}D_P-\sum\limits_{k=1}^{\infty}D_PP^{k-1}D_{P^*}^2P^{*(m+k-2)}S_{n-j}^*D_P\\
=&D_PP^*S_{n-j}^*P^{*(m-2)}D_P-\sum\limits_{k=1}^{\infty}D_P\left(P^{k-1}P^{*(k-1)}-P^kP^{*k}\right)P^{*(m-1)}S_{n-j}^*D_P\\
=&D_PS_{n-j}^*P^{*(m-1)}D_P-D_P(I-\mathcal{A}_*)P^{*(m-1)}S_{n-j}^*D_P\\
=&D_P\mathcal{A}_*P^{*(m-1)}S_{n-j}^*D_P.
\end{align*}
\endgroup
\textbf{Coefficient of $e^{imt}$, $m\leq -1$:}
\begingroup
\allowdisplaybreaks
\begin{align*}
&D_PP^{-m-1}D_{P^*}B_j^*P-D_PP^{-m}D_{P^*}(B_j^*D_{P^*}D_P-B_{n-j}P)\\
&-\sum\limits_{k=2-m}^{\infty}D_PP^{k-1}D_{P^*}^2P^{*(m+k-2)}S_{n-j}^*D_P\\
=&D_PP^{-m-1}D_{P^*}PA_j-D_PP^{-m}(S_j^*-S_{n-j}P^*)^*D_P+D_PP^{-m}D_{P^*}B_{n-j}P\\
&-\sum\limits_{k=2-m}^{\infty}D_PP^{1-m}\left(P^{m+k-2}P^{*(m+k-2)}-P^{m+k-1}P^{*(m+k-1)}\right)S_{n-j}^*D_P\\
=&D_PP^{-m}D_PA_j-D_PP^{-m}S_jD_P+D_PP^{-m+1}S_{n-j}^*D_P\\
&+D_PP^{-m}D_{P^*}B_{n-j}P-D_PP^{1-m}(I-\mathcal{A}_*)S_{n-j}^*D_P\\
=&D_PP^{-m}D_PA_j-D_PP^{-m}S_jD_P+D_PP^{-m+1}S_{n-j}^*D_P\\
&+D_PP^{-m}D_{P^*}B_{n-j}P-D_PP^{-m+1}S_{n-j}^*D_P+D_PP^{-m+1}\mathcal{A}_*S_{n-j}^*D_P\\
=&D_{P}P^{-m}(D_PA_j+D_{P^*}B_{n-j}P)-D_PP^{-m}S_jD_P+D_PP^{-m+1}\mathcal{A}_*S_{n-j}^*D_P\\
=& D_PP^{-m+1}\mathcal{A}_*S_{n-j}^*D_P, \quad(\text{using Lemma }\ref{BPlem1}).
\end{align*}
\endgroup
Thus Equation \eqref{expression2} holds.



\begin{thebibliography}{99}
	
\bibitem{A:W:Y} A. A. Abouhajar, M. C. White and N. J. Young, A Schwarz
lemma for a domain related to $\mu$-synthesis, \textit{J. Geom.
	Anal.} 17 (2007) 717 - 750.

\bibitem{Agler} J. Agler and N. J. Young, A commutant lifting theorem for a domain in $\Bbb C^2$ and spectral interpolation, \textit{J. Funct. Anal.}, 161 (1999) 452 - 477.

\bibitem{wermer} H. Alexander and J. Wermer,
Several complex variables and Banach algebras, \textit{Graduate
	Texts in Mathematics, 35; 3rd Edition}, Springer, (1997).


\bibitem{Nagy} H. Bercovici, C. Foias, L. Kerchy and B. Sz.-Nagy, Harmonic Analysis of Operators on Hilbert Space, \textit{Universitext, Springer}, New York, 2010.

\bibitem{T:B} T. Bhattacharyya, The tetrablock as a spectral
set, \textit{Indiana Univ. Math. J.}, 63 (2014), 1601 - 1629.

\bibitem{TirthaHari} T. Bhattacharyya, S. Lata and H. Sau, Admissible fundamental operators, \textit{J. Math. Anal. Appl.} 425 (2015), 983 - 1003.

\bibitem{PalAdv} T. Bhattacharyya, S. Pal and S. Shyam Roy, Dilations of $\Gamma$-contractions by solving operator equations, \textit{Adv. Math.}, 230 (2012), 577 - 606.

\bibitem{B:P} B. Bisai and S. Pal, The fundamental operator tuples associated with the symmetrized polydisc, \textit{New York J. Math.}, 27 (2021), 349 - 362.

\bibitem{B:P5} B. Bisai and S. Pal, A model theory for operators associated with a domain related to $\mu$-synthesis, \textit{Collect. Math.} (To appear), https://doi.org/10.1007/s13348-021-00341-6.

\bibitem{B:P6} B. Bisai and S. Pal, A Nagy-Foias program for a c.n.u $\Gamma_{n}$-contractions, \textit{Preprint}, arXiv:2110.03436.

\bibitem{S:B} S. Biswas and Subrata S. Roy, Functional models of $\Gamma_n$-contractions and characterization of $\Gamma_n$-isometries, \textit{J. Funct. Anal} 266, 6224 - 6255 (2014).

\bibitem{costara} C. Costara, On the spectral Nevanlinna-Pick problem, \textit{Stud. Math.} 170 (2005) 23 - 55.

\bibitem{curto} R. E. Curto, Applications of several complex variables to multiparameter spectral theory in: Surveys of Some Recent Results in Operator Theory, Vol. II, in: Pitman Res. Notes Math. Ser., vol. 192, Longman Sci. Tech., Harlow, 1988, pp. 25 - 90.

\bibitem{Douglas} R. G. Douglas, Structure theory for operators. I., \textit{J. Reine Angew. Math.} 232 (1968) 180 - 193.

\bibitem{doyel} J. Doyel, ANalysis of feedback systems with structured uncertainties, \textit{IEE Proc.}, Control Theory Appl. 129 (1982) 242 - 250.

\bibitem{E:Z} A. Edigarian and W. Zwonek, Geometry of the symmetrized polydisc, \textit{Arch. Math.} 84 (2005) 364 - 374.

\bibitem{Kubrusly} Carlos S. Kubrusly, An Introduction to Models and Decompositions in Operator Theory, \textit{Birkhauser Boston}, 1997.

\bibitem{A:Pal} A. Pal, On $\Gamma_n$-contractions and their conditional dilations, \textit{J. Math. Anal. Appl.}, 510 (2022), no. 2, Paper No. 126016.

\bibitem{S:Pdecom} S. Pal, Canonical decomposition of operators associated
with the symmetrized polydisc, \textit{Complex Anal. Oper. Theory}, 12 (2018),
931 - 943.

\bibitem{S:P} S. Pal, Operator theory and distinguished varieties in the symmetrized $n$-disk, \textit{https://arxiv.org/abs/1708.06015}.

\bibitem{SPrational} S. Pal, Rational dilation for operators associated with spectral interpolation and distinguished varieties, \textit{https://arxiv.org/abs/1712.05707}.

\bibitem{PalNagy} S. Pal, Dilation, functional model and a complete unitary invariant for $C_{.0}$ $\Gamma_{n}$-contractions, \textit{https://arxiv.org/abs/1708.06015}.
 
\bibitem{Taylor} J. L. Taylor, The analytic -functional calculus for several commuting opoerators, \textit{Acta Math.} 125 (1970) 1 - 38.

\bibitem{Taylor1} J. L. Taylor, A joint spectrum for several commuting operators, \textit{J. Func. Anal.} 6 (1970) 172 - 191.

\end{thebibliography}
\end{document}